\documentclass{amsart}
\usepackage{wasysym}
\usepackage{amsmath}
\usepackage{mathtools}
\usepackage{amssymb}
\usepackage{amsthm}
\usepackage{tikz}
\usetikzlibrary{matrix,arrows,backgrounds,shapes.misc,shapes.geometric,patterns,calc,positioning}
\usetikzlibrary{calc,shapes,babel}
\usepackage{wrapfig}
\usepackage{epsfig}  
\usepackage[margin=1.3in]{geometry}
\usepackage{color}	 %
\input{xy}
\xyoption{poly}
\xyoption{2cell}
\xyoption{all}
\newtheorem{thm}{Theorem}[section]
\newtheorem{prop}[thm]{Proposition}
\newtheorem{lemma}[thm]{Lemma}
\newtheorem{corollary}[thm]{Corollary}

\newtheorem*{thm*}{Theorem}

\usepackage{hyperref}
\theoremstyle{definition}
\newtheorem{definition}[thm]{Definition}
\theoremstyle{remark}

\newtheorem{example}[thm]{Example}
\numberwithin{equation}{section}

\newcommand{\Hom}{\textup{Hom}}
\newcommand{\Fac}{\textup{Fac}}
\newcommand{\MGS}{\textup{MGS}}
\newcommand{\add}{\textup{add}}
\newcommand{\proj}{\textup{proj}}
\newcommand{\projinj}{\textup{proj-inj}}
\newcommand{\coker}{\textup{coker}}
\newcommand{\im}{\textup{im}}

\newcommand{\pd}{\textup{pd}}

\newcommand{\soc}{\textup{soc}}
\newcommand{\tp}{\textup{top}}
\newcommand{\Ext}{\textup{Ext}}
\newcommand{\End}{\textup{End}}

\newcommand{\stilt}{\textup{s}\tau\textup{-tilt}}
\newcommand{\tilt}{\textup{tilt}}
\renewcommand{\mod}{\textup{mod}}

\begin{document}

\title{Classical tilting and $\tau$-tilting theory via duplicated algebras} 
\author{Jonah Berggren}
\address{Department of Mathematics, University of Kentucky, Lexington, KY 40506-0027, USA}
\email{jrberggren@uky.edu}

\author{Khrystyna Serhiyenko}
\thanks{The authors were
supported by the NSF grant DMS-2451909. This work was supported by a grant from the
Simons Foundation International [SFI-MPS-TSM-00013650, KS]}
\address{Department of Mathematics, University of Kentucky, Lexington, KY 40506-0027, USA}
\email{khrystyna.serhiyenko@uky.edu}


\maketitle
\setcounter{tocdepth}{2}

\begin{abstract} 
$\tau$-tilting theory can be thought of as a generalization of the classical tilting theory which allows mutations at any indecomposable summand of a support $\tau$-tilting pair.  Indeed, for any algebra $\Lambda$ its tilting modules $\tilt\,\Lambda$ form a subposet of the support $\tau$-tilting poset $\stilt\,\Lambda$.  We show that conversely the $\tau$-tilting theory of an algebra $\Lambda$ can be naturally identified with the classical tilting theory of its duplicated algebra $\bar\Lambda$ by establishing a poset isomorphism $\stilt\,\Lambda\cong \tilt\,\bar\Lambda$.  As a result, $\tau$-tilting theory may be considered to be a special case of tilting theory.  This extends the results of Assem-Br\"ustle-Schiffler-Todorov in the case of hereditary algebras.  We also show that the product $\stilt\,\Lambda\times \stilt\,\Lambda$ embeds into the support $\tau$-tilting poset of its duplicated algebra $\stilt\,\bar\Lambda$ as a collection of Bongartz intervals.  As an application we obtain a similar inclusion on the level of maximal green sequences. 
\end{abstract}

\section{Introduction}

Tilting theory is powerful tool for relating different categories and it appears in many areas of mathematics, including representation theory, commutative and non-commutative algebraic geometry, and algebraic topology.  It originated with the work of Bernstein, Gelfand, and Ponomarev on reflection functors \cite{BGP} in 1973, relating representations of two quivers.  Reflection functors also provided important insights into Gabriel's classification of quivers of finite representation type.  They were subsequently generalized in the work of Auslander, Platzeck, and Reiten \cite{APR} and in the Brenner-Butler tilting theorem \cite{BB}.   The present definition of a tilting module first appeared in the work of Happel and Ringel \cite{HR}, while the study of the tilting poset of an algebra and its combinatorics was lead by Unger \cite{U}.   Since then, tilting theory has been vastly generalized from the setting of modules over finite dimensional algebras to general rings, and further extended to the study of co-tilting theory, silting theory, and cluster tilting theory. For more information on tilting theory we refer to \cite{ASS, HHK} and references therein. 

Another important generalization of tilting theory is the $\tau$-tilting theory introduced in 2014 in the seminal work of  Iyama, Adachi, and Reiten \cite{AIR}.   They show that for a given finite dimensional algebra, $\tau$-tilting theory is equivalent to the study of 2-term silting complexes of projective modules as well as functorially finite torsion classes.  Moreover, cluster algebras of Fomin and Zelevinsky \cite{FZ} can be categorified in terms of the $\tau$-tilting theory of certain 2-Calabi-Yau tilted algebras, hence $\tau$-tilting theory uncovers cluster-like combinatorics in the module category of every finite dimensional algebra.   Furthermore, the development of $\tau$-tilting theory motivated many new results in representation theory.  In particular, there are close connections between the $\tau$-tilting theory and the study of bricks \cite{DIRRT}, stability conditions \cite{As}, g-vectors \cite{DIJ}, c-vectors \cite{F}, and green sequences \cite{BST}.  

Given any finite dimensional algebra $\Lambda$, a tilting module naturally yields a support $\tau$-tilting pair, thus $\tau$-tilting theory naturally contains tilting theory of an algebra.  More specifically, tilting modules form a poset $\tilt\,\Lambda$ where two tilting modules are connected by an edge in the Hasse diagram of $\tilt\,\Lambda$ whenever they differ by a single indecomposable summand.   Exchanging one indecomposable summand of a tilting module for another is called a mutation.   Similarly support $\tau$-tilting modules form a poset $\stilt\,\Lambda$, and  $\tilt\,\Lambda$ is a cover--preserving subposet of $\stilt\,\Lambda$.   Hence, $\tau$-tilting theory is often viewed as a completion of the classical tilting theory which allows mutations at every indecomposable summand. 

In this article, we present the opposite point of view by showing that for any finite dimensional algebra $\Lambda$ its $\tau$-tilting theory is equivalent to the tilting theory of its duplicated algebra $\bar\Lambda$.  The duplicated algebra of $\Lambda$ is defined as the $2\times 2$ triangular matrix algebra with $\Lambda$ on the diagonal and the dual bimodule $D\Lambda$ on the subdiagonal.  Our first main result is as follows. 

\begin{thm}\label{thm:intro1}
Let $\Lambda$ be an algebra and let $\bar\Lambda$ denote its duplicated algebra.  Then there is an isomorphism of posets 
\[\stilt\,\Lambda  \cong \tilt\,\bar\Lambda.\]
\end{thm}

This isomorphism is achieved by an explicit map $\bar F$ that maps projective presentations of support $\tau$-tilting $\Lambda$-modules to projective resolutions of tilting $\bar\Lambda$-modules.  We note that for hereditary algebras $\Lambda$, this result can be derived from the 2006 work of Assem, Br\"ustle, Schiffler, and Todorov \cite{ABST}, who show that tilting modules over $\bar\Lambda$ are in bijection with cluster tilting objects in the cluster category $\mathcal{C}_{\Lambda}$ of $\Lambda$ from \cite{BMRRT}.  Note that the latter are in bijection with support $\tau$-tilting $\Lambda$-modules by \cite{AIR}.   Hence, Theorem~\ref{thm:intro1} generalizes the results of \cite{ABST} beyond the hereditary case.  We also note that the relation between tilting modules of an algebra and its duplicated algebra, or more generally $m$-replicated algebras, has been studied in a number of works by Zhang \cite{Z, Z2}, and in relation to higher Auslander-Reiten theory by Chan, Iyama, and  Marczinzik \cite{CIM}.  In another direction, recently Hafezi, Nasr-Isfahani, and Wei establish a bijection between $\stilt\,\Lambda$ and tilting objects in the category of morphisms of projective $\Lambda$-modules \cite{HNIW}.  

Next, we further relate support $\tau$-tilting posets of $\Lambda$ and $\bar\Lambda$.  In particular, we show that the product of support $\tau$-tilting posets of $\Lambda$ is a subposet of the support $\tau$-tilting poset of $\bar\Lambda$.

\begin{thm}\label{thm:intro2}
Let $\Lambda$ be an algebra and let $\bar\Lambda$ denote its duplicated algebra.   Then there is an inclusion of posets 
\[\varphi: \stilt\,\Lambda \times \stilt\,\Lambda  \to \stilt\,\bar\Lambda. \]
\end{thm}

The definition of the map $\varphi$ uses the map $\bar F$ from Theorem~\ref{thm:intro1} and has the following property.  Fixing a support $\tau$-tilting $\Lambda$ module $M$ in the first component of the map $\varphi$ while letting the second component vary over $\stilt\,\Lambda$,  the image of $\varphi (M, -)$ inside $\stilt\,\bar\Lambda$ is a Bongartz interval that is isomorphic to $\stilt\,\Lambda$.   Hence, $\stilt\,\bar\Lambda$ contains many copies of $\stilt\,\Lambda$, which are all Bongartz intervals.   The key ingredient in proving this theorem is the $\tau$-tilting reduction of Jasso \cite{Jasso}.

A maximal green sequence for an algebra $\Lambda$ is a finite maximal chain in the poset $\stilt\,\Lambda$. 
Maximal green sequences were first introduced by Keller in \cite{K} in the context of cluster algebras to compute refined Donaldson-Thomas invariants.  They are also important in constructing twist automorphisms and proving existence of theta bases in cluster algebras. The notion of maximal green sequences was generalized to module categories of finite dimensional algebras, for example see the survey \cite{KD} and references therein. Moreover,  one of the main questions in this topic is whether an algebra admits a maximal green sequence.  

Let $\MGS(\Lambda)$ denote the set of all maximal green sequence of $\Lambda$.  As an application of Theorem~\ref{thm:intro2} we conclude that a pair of maximal green sequences for an algebra $\Lambda$ induces a maximal green sequence for $\bar\Lambda$ that lies inside the image of $\varphi$.   This result implies that the image of the map $\varphi$ spans the entire length of the support $\tau$-tilting poset of $\bar\Lambda$.

\begin{corollary}\label{cor:intro}
There is an inclusion of maximal green sequences 
\[\psi: \,\,\,   \MGS(\Lambda) \times \MGS(\Lambda)  \to \MGS(\bar\Lambda).\]
\end{corollary}

The article is organized as follows.  In Section 2 we recall the relevant background on tilting and $\tau$-tilting theory as well as duplicated algebras.  In Section 3 we define the map $F: \mod\,\Lambda \to \mod\,\bar\Lambda$ and describe its image.  We then use $F$ to construct a bijection between $\stilt\,\Lambda$ and $\tilt\,\bar\Lambda$ in Section 4.   In the next section, we define another map $\bar G: \stilt\,\Lambda\to \stilt\,\bar\Lambda$ whose image is the Bongartz interval in $\stilt\,\bar\Lambda$ defined by $\Lambda$.  We then combine the two maps $\bar F, \bar G$ to obtain inclusions in Theorem~\ref{thm:intro2} and Corollary~\ref{cor:intro}.

\section{Background}

Let $\Lambda$ be a finite dimensional algebra over an algebraically closed field $k$.  Let $\text{mod}\,\Lambda$ denote the category of finite dimensional right $\Lambda$ modules. Let $\proj\,\Lambda, \projinj\,\Lambda$ denote the category of projective and projective-injective $\Lambda$-modules, respectively. Let $P_\Lambda(i), I_\Lambda(i), S_{\Lambda}(i)$ denote the indecomposable projective, injective, and simple module at vertex $i$, respectively. Given $M\in\mod\,\Lambda$, let $\left | M \right |$ be the number of nonisomorphic indecomposable summands of $M$, and $\pd_{\Lambda} M$ be its projective dimension.  A module $M$ is said to be basic if all of its indecomposable summands are nonisomorphic. 
Furthermore, $\Fac \, M$ is the set of all quotients of direct sums of $M$, and $\add\, M$ is the additive subcategory of $\mod\,\Lambda$ consisting of all direct sums of summands of $M$.   We let $\Omega_{\Lambda} M, \Omega_{\Lambda}^{-1} M$ be the syzygy and the cosyzygy of $M$.  
Lastly, we denote the standard duality functor by $D=\Hom_k(-, k)$ and the Auslander-Reiten translation by $\tau_{\Lambda}$.

\subsection{Tilting}
We recall some definitions and results from the theory of tilting and $\tau$-tilting.    We begin with the classical tilting theory. 

\begin{definition}
A basic module $M\in \text{mod}\,\Lambda$ is called \emph{tilting} if it satisfies all of the following conditions:
\begin{itemize}
\item[(a)] $M$ is rigid, that is $\Ext^1_{\Lambda}(M, M)=0$, 
\item[(b)] $\pd_{\Lambda} M\leq 1$, 
\item[(c)] $\left | M\right | = \left | \Lambda\right |$. 
\end{itemize}
Moreover, it is called \emph{partial tilting} if it satisfies (a) and (b). 
\end{definition}

Note that any partial tilting module can be completed to a tilting module by adding extra summands as shown by Bongartz \cite{B}. 

Let $\tilt\,\Lambda$ denote the set of all tilting $\Lambda$-modules up to isomorphisms.  Moreover, $\tilt\,\Lambda$ is equipped with a poset structure where the partial order is defined as follows: $M\leq M' \in \tilt\,\Lambda$ if $\Fac \, M \subseteq \Fac \, M'$.

Next, we review the relevant notions of $\tau$-tilting theory introduced by Adachi, Iyama, and Reiten following \cite{AIR}.  We also draw the reader's attention to the various parallels between tilting and $\tau$-tilting.  

\begin{definition}
A pair of basic modules $(M, P_{\Lambda})$ with $M\in \mod\,\Lambda, P_{\Lambda}\in\proj \,\Lambda$ is called \emph{support $\tau$-tilting} if it satisfies all of the following conditions:
\begin{itemize}
\item[(a)] $M$ is \emph{$\tau$-rigid}, that is $\Hom_{\Lambda}(M, \tau_{\Lambda}M)=0$, 
\item[(b)] $\Hom_{\Lambda}(P_{\Lambda}, M)=0$, 
\item[(c)] $\left | M\oplus P_{\Lambda}\right | = \left | \Lambda\right |$.
\end{itemize}
Moreover, a pair $(M, P_{\Lambda})$ is called \emph{$\tau$-rigid} if it satisfies (a) and (b).
\end{definition}

Note that a support $\tau$-tilting pair $(M, P_{\Lambda})$ is determined by $M$, where $M$ is called a \emph{support $\tau$-tilting module}.  Hence, we may sometimes omit projective modules in the second position and only focus on support $\tau$-tilting modules.   

Let $\stilt\,\Lambda$ denote the set of all support $\tau$-tilting pairs up to isomorphisms (meaning that $(M, P)\cong (M', P')$ if and only if $M\cong M$ and $P\cong P'$), or equivalently support $\tau$-tilting modules up to isomorphisms, in $\mod\,\Lambda$.  Moreover, $\stilt\,\Lambda$ is a poset where the partial order is defined as follows: $(M, P_{\Lambda})\leq (M', P_{\Lambda}') \in \stilt\,\Lambda$ if $\Fac \, M \subseteq \Fac \, M'$.  We observe that every tilting module $M$ is a support $\tau$-tilting module, or equivalently induces a support $\tau$-tilting pair $(M, 0)$, see \cite[Proposition 2.2]{AIR}.   This induces a cover-preserving embedding (an embedding of posets that preserves cover relations) from $\tilt \,\Lambda$ to $\stilt\,\Lambda$.    Hence $\tilt \,\Lambda$ is a cover-preserving subposet of $\stilt\,\Lambda$, and in this way support $\tau$-tilting theory is considered to be a generalization of tilting theory.

A $\tau$-rigid pair $(M, P_{\Lambda})$ can always be completed to a support $\tau$-tilting pair $(M\oplus M', P_{\Lambda}\oplus P_{\Lambda}')$ by adding extra summands.  Such a completion is not unique.  Let $\stilt_{(M, P_{\Lambda})} \Lambda$ denote the subset of $\stilt\,\Lambda$ consisting of all possible completions of a $\tau$-rigid pair $(M, P_{\Lambda})$; that is, 
\[ \stilt_{(M, P_{\Lambda})} \Lambda = \{ (M', P_{\Lambda}')\in \stilt\,\Lambda \mid M\in \add\, M', P_{\Lambda}\in \add \, P_{\Lambda}'\}.\]
  It is known that $\stilt_{(M, P_{\Lambda})} \Lambda$ forms an interval in $\stilt\,\Lambda$ called the \emph{Bongartz interval} of $(M, P_{\Lambda})$.  Moreover, the largest element of $\stilt_{(M, P_{\Lambda})} \Lambda$ is called the \emph{Bongartz completion} of $(M, P_{\Lambda})$, and it is of the form   $(M\oplus M', P_{\Lambda})$.  Similarly, the smallest element of $\stilt_{(M, P_{\Lambda})} \Lambda$ is called the \emph{Bongartz co-completion} of $(M, P_{\Lambda})$.  Note that the Bongartz co-completion of $(M, P_{\Lambda})$ may not be of the form $(M, P_{\Lambda}\oplus P')$ for some projective module $P'$. 

We also recall the notion of a torsion pair. A \emph{torsion pair} is a pair of full subcategories $(\mathcal{T}, \mathcal{F})$ in $\mod\,\Lambda$ such that $\Hom_{\Lambda}(T, F)=0$ for all $T\in \mathcal{T}, F\in \mathcal{F}$, and moreover $\mathcal{T}, \mathcal{F}$ are maximal with respect to this property.  Note that in this way $\mathcal{T}$ determines  $\mathcal{F}$.  Furthermore, every $\tau$-rigid module $T$ induces a torsion pair $(\mathcal{T}(T), \mathcal{F}(T))$ where $\mathcal{T}(T)=\Fac\, T$.

Given a torsion pair $(\mathcal{T}, \mathcal{F})$ and any $M\in\mod\,\bar\Lambda$ there exists a \emph{canonical sequence}, unique up to isomorphisms, 
\begin{equation}\label{eq:cs}
0\to tM \to M \to fM \to 0
\end{equation}
with $tM\in \mathcal{T}$ and $fM\in \mathcal{F}$, see for example \cite[Proposition VI. 1.5]{ASS}.  In this case $tM$ and $fM$ are called the \emph{torsion part} of $M$  and the \emph{torsion free part} of $M$ corresponding to the torsion pair $(\mathcal{T}, \mathcal{F})$.  For additional information on tilting theory we refer to \cite[Chapter VI]{ASS}.

Lastly, we review $\tau$-tilting reduction, an important technique in $\tau$-tilting theory developed by Jasso \cite{Jasso}.  Let $(M, P_{\Lambda})$ be a $\tau$-rigid pair in $\mod\,\Lambda$.  It defines a torsion pair $(\Fac\, M, M^\perp)$, where 
\[M^\perp = \{X \in \mod\,\Lambda\mid \Hom_{\Lambda}(M, X)=0\}.\]

Let $(M\oplus M', P_{\Lambda})$ be the Bongartz completion of   $(M, P_{\Lambda})$, and define  $B_{(M, P_{\Lambda})} = \End_{\Lambda}(M\oplus M')$ to be the endomorphism algebra of $M\oplus M'$.  
Let 
\[C_{(M, P_{\Lambda})}=B_{(M, P_{\Lambda})}/\left<e_M\right>\] 
be the quotient of $B_{(M, P_{\Lambda})}$ by the idempotent $e_M$ corresponding to the projective $B_{(M, P_{\Lambda})}$-module $\Hom_{\Lambda}(M\oplus M', M)$.  With this notation one of the main results of \cite{Jasso} can be stated as follows. 

\begin{thm}\cite[Theorem 3.16]{Jasso}\label{thm:red}
Let $(M, P_{\Lambda})$ be a $\tau$-rigid pair in $\mod\,\Lambda$.  Then there is an isomorphism of posets defined on support $\tau$-tilting modules 
\begin{align*}
\stilt _{(M, P_{\Lambda})} \Lambda &\to \stilt \,C_{(M, P_{\Lambda})} \\
N & \mapsto \Hom_{\Lambda}(M\oplus M', fN)
\end{align*}
where $fN\in M^\perp$ is defined by the canonical sequence \eqref{eq:cs} corresponding to the torsion pair $(\Fac\, M, M^\perp)$. 
\end{thm}

\subsection{Duplicated algebra}
We recall the construction and the main properties of a duplicated algebra, which is one of the main objects of study of this paper. 
Given an algebra $\Lambda$, the triangular matrix algebra 
\[\bar\Lambda = \begin{bmatrix}\Lambda & 0 \\ D\Lambda & \Lambda\end{bmatrix}\] 
is called the \emph{duplicated algebra of $\Lambda$}.  Observe that $\bar\Lambda$ contains two copies of the algebra $\Lambda$, which we will denote by $\Lambda$ and $\Lambda^\bullet$ such that $\Lambda^\bullet$ corresponds to the bottom right corner of the matrix representation of $\bar\Lambda$.  
Thus, the primitive orthogonal idempotents of $\bar\Lambda$ are the primitive idempotents of $\Lambda$ and $\Lambda^\bullet$.  Accordingly, every vertex $i$ of the quiver of $\Lambda$ yields two vertices of $\bar\Lambda$, which we will denote by $i, i^\bullet$.   Also, note that every $\Lambda$-module is also a module over $\bar\Lambda$, thus $\text{mod}\,\Lambda, \text{mod}\,\Lambda^\bullet$  are full subcategories of $\text{mod}\,\bar\Lambda$. 

The following statement describing the indecomposable projective and injective $\bar\Lambda$-modules follows from \cite{ABST}.  For more information on duplicated algebras see \cite{A, ANS}.

\begin{lemma}\cite[Section 1.3]{ABST}\label{lem:proj}
Let $\bar\Lambda$ be the duplicated algebra of an algebra $\Lambda$, then the following statements hold:
\begin{itemize}
\item[(a)] $P_{\bar\Lambda}(i)\cong P_{\Lambda}(i)$,
\item[(b)] $P_{\bar\Lambda}(i^\bullet)\cong I_{\bar\Lambda}(i)$, and there exists a short exact sequence $0\to I_{\Lambda}(i)\to P_{\bar\Lambda}(i^\bullet)\to P_{\Lambda^\bullet} (i^\bullet)\to 0$ in $\textup{mod}\,\bar\Lambda$,
\item[(c)] $I_{\bar\Lambda}(i^\bullet)\cong I_{\Lambda^\bullet}(i^\bullet)$.
\end{itemize}
\end{lemma}

\begin{example}\label{ex:2}
If $\Lambda = k(1\to 2)$ then $\bar\Lambda$ is the path algebra of the following quiver with relation $\alpha^\bullet \beta \alpha=0$.  
\[\xymatrix{
1^\bullet \ar[r]^{\alpha^\bullet} & 2^\bullet  \ar[dl]_{\beta}\\
1\ar[r]^{\alpha} & 2
}\]
We see that $\bar\Lambda$ contains $\Lambda = k(1\xrightarrow{\alpha} 2)$ and $\Lambda^\bullet = k(1^\bullet \xrightarrow{\alpha^\bullet} 2^\bullet)$, and their corresponding module categories $\text{mod}\,\Lambda, \text{mod}\,\Lambda^\bullet$  are full subcategories of $\text{mod}\,\bar\Lambda$. 
Note that the indecomposable projective $\bar\Lambda$-modules are projective $\Lambda$-modules, $P(1)=\begin{smallmatrix}1\\2\end{smallmatrix}$ and $P(2)=\begin{smallmatrix}2\end{smallmatrix}$ together with the projective-injective $\bar\Lambda$-modules $P(1^\bullet) = \begin{smallmatrix}1^\bullet\\2^\bullet\\1\,\,\,\end{smallmatrix}$ and $P(2^\bullet) = \begin{smallmatrix}2^\bullet\\1\,\,\,\\2\,\,\,\end{smallmatrix}$. 
\end{example}

\section{Construction of the map $F$}

In this section we construct a map $F$ that maps $\Lambda$-modules to $\bar\Lambda$-modules.  We show that it is injective and describe its image inside $\text{mod}\,\bar\Lambda$ as nearly all $\bar\Lambda$-modules of projective dimension at most one. 

\begin{definition}
Define a map on modules
\begin{align*}
F: \textup{mod}\,\Lambda &\to \textup{mod}\,\bar\Lambda\\
M &\mapsto F(M)
\end{align*}
as follows.  Consider the following commutative diagram 

\begin{equation}\label{eq:def}
\xymatrix{
0\ar[r]&P^1_{\Lambda}\ar[r]^{\begin{bsmallmatrix}f\\g\end{bsmallmatrix}} \ar@{=}[d]& P^0_{\Lambda}\oplus P_{\bar\Lambda}' \ar[r] \ar[d]^{\begin{bsmallmatrix}1 & 0\end{bsmallmatrix}}& F(M)\oplus P_{\bar\Lambda}''\ar[r] \ar[d]& 0\\
&P^1_{\Lambda} \ar[r]^f& P^0_{\Lambda} \ar[r] & M \ar[r] & 0}
\end{equation}
where the bottom row is a projective presentation of $M$ in $\textup{mod}\,\Lambda$ and the top row is constructed as described below. 
Let $g: P_{\Lambda}^1\to P_{\bar\Lambda}'$ denote an injective envelope of $P^1_{\Lambda}$ in $\text{mod}\,\bar\Lambda$, and note that $P_{\bar\Lambda}'$ is projective-injective by Lemma~\ref{lem:proj}(b).   Define $F(M)$ to be the maximal summand of the cokernel of ${\begin{bsmallmatrix}f\\g\end{bsmallmatrix}}$ containing no nonzero summands of $\projinj\,\bar\Lambda$.   Thus, the cokernel is of the form $F(M)\oplus P_{\bar\Lambda}''$ where  $P_{\bar\Lambda}''$ is a maximal projective-injective summand of the cokernel which may possibly be zero. 

Note that the square on the left commutes by construction, so by the universal property of the cokernel there exists a map $F(M)\oplus P_{\bar\Lambda}''\to M$ that makes the second square commute.  
\end{definition}

Note that $F$ is only a map on the modules, and we do not define it on morphisms.  The next diagram illustrates a more precise relationship between $M$ and $F(M)$, which will play an important role in the forthcoming discussion.  

\begin{lemma}\label{lem:33}
Given $M\in\textup{mod}\,\Lambda$ there exists a commutative diagram in $\textup{mod}\,\bar\Lambda$ with exact rows and columns

\begin{equation}\label{eq:33}
\xymatrix{
& 0 \ar[d] & 0 \ar[d] & 0 \ar[d]\\
0\ar[r]& \Omega^2_{\Lambda} M \ar[d]\ar[r] & P_{\bar\Lambda} \ar[r] \ar[d]^{\begin{bsmallmatrix}0 \\ 1\end{bsmallmatrix}} & \ker\,\pi_M \ar[d]\ar[r] & 0\\
0\ar[r]&P^1_{\Lambda}\ar[r] \ar[d]& P^0_{\Lambda}\oplus P_{\bar\Lambda} \ar[r] \ar[d]^{\begin{bsmallmatrix}1 & 0\end{bsmallmatrix}}& F(M)\ar[r] \ar[d]^{\pi_M}& 0\\
0\ar[r]&\Omega_{\Lambda} M\ar[r] \ar[d]& P^0_{\Lambda} \ar[r] \ar[d]& M \ar[r]\ar[d] & 0\\
& 0 & 0 & 0}
\end{equation}
where $P^0_\Lambda$ is a projective cover of $M$,  $P^0_{\Lambda}\oplus P_{\bar\Lambda}$ is a projective cover of $F(M)$, and $P_{\bar\Lambda}$ is a projective-injective $\bar\Lambda$-module.  In particular, the middle row is a minimal projective resolution of $F(M)$ in $\textup{mod}\,\bar\Lambda$ and $\ker\,\pi_M \cong \Omega^{-1}_{\bar\Lambda} \Omega^2_{\Lambda} M$.  
\end{lemma}

\begin{proof}
Recall the construction of the module $F(M)$ and the diagram \eqref{eq:def}.  We will show how to obtain \eqref{eq:33} from \eqref{eq:def}.  Observe that the bottom row of \eqref{eq:33} comes from the bottom row of \eqref{eq:def} by making the sequence exact on the left.  In particular, $P^0_\Lambda$ is the projective cover of $M$ in $\text{mod}\,\Lambda$, and the kernel of the projection $P^0_\Lambda\to M$ is the syzygy of $M$ denoted by $\Omega_\Lambda M$.  

Now consider the top sequence in \eqref{eq:def}. By definition it is a projective resolution of $F(M)\oplus P_{\bar\Lambda}''$, where $P_{\bar\Lambda}''$ is a projective-injective $\bar\Lambda$-module.  Then $P_{\bar\Lambda}''$ is a direct summand of $P_{\Lambda}^0\oplus P_{\bar\Lambda}'$.  Since $P_{\Lambda}^0$ is a projective $\Lambda$-module, then it does not contain any summands of $\projinj\,\bar\Lambda$.  Thus, $P_{\bar\Lambda}''$ is a direct summand of $P_{\bar\Lambda}'$.  Let $P_{\bar\Lambda}$ be obtained from $P_{\bar\Lambda}'$ by removing $P_{\bar\Lambda}''$.  This yields the short exact sequence appearing in the second row of  \eqref{eq:33}.  Note that this sequence is a projective resolution of $F(M)$, which is moreover minimal.  Indeed, $P_\Lambda^1$ does not contain any summands in common with $P_{\bar\Lambda}$, so if $P^1_{\Lambda}\to P^0_{\Lambda}\oplus P_{\bar\Lambda}$ were not minimal then neither would $P_{\Lambda}^1\to P_{\Lambda}^0$ be minimal.  The latter is a minimal projective resolution of $M$, which would be a contradiction.   Lastly, from  \eqref{eq:def} we also obtain the induced vertical maps in \eqref{eq:33} from the second row to the third that make the two squares commute.  Note that these vertical maps are indeed surjective. 

The top row of \eqref{eq:33} is obtained by taking the kernel of the vertical maps, which yields an exact sequence by the snake lemma.  Note that by definition $P_{\Lambda}^1$ is the projective cover of $\Omega_{\Lambda} M$, hence the kernel of the associated projection is $\Omega^2_{\Lambda} M$ as claimed.  

Finally, observe that if $\ker\,\pi_M$ were to contain any projective-injective $\bar\Lambda$-modules, then they would also have to be summands of $F(M)$, because of the short exact sequence appearing the third column of the diagram, contradicting the definition of the map $F$.  Hence, $\ker\,\pi_M$ does not contain any projective-injective summands, which means that $P_{\bar\Lambda}$ is an injective envelope of $\Omega^2_{\Lambda} M$ in $\text{mod}\,\bar\Lambda$, so the cokernel $\ker\,\pi_M$ of the inclusion $\Omega^2_{\Lambda}M \to P_{\bar\Lambda}$ is isomorphic to $\Omega^{-1}_{\bar\Lambda} \Omega^2_{\Lambda} M$.   This completes the proof of the lemma. 
\end{proof}

Next, we obtain some important properties of the map $F$.

\begin{prop}\label{prop:F}
Let $M, N\in \textup{mod}\,\Lambda$ then the following statements hold: 
\begin{itemize}
\item[(a)] $F(M\oplus N)\cong F(M)\oplus F(N)$,
\item[(b)] $\textup{pd}_{\bar\Lambda} F(M)\leq 1$,
\item[(c)] $F(M)\cong M$ if and only if $\textup{pd}_{\Lambda} M \leq 1$.  
\item[(d)] $M$ is indecomposable if and only if $F(M)$ is indecomposable.
\end{itemize}
\end{prop}

\begin{proof}
Part (a) follows by definition of $F$, since projective presentation of a direct sum of modules is a direct sum of projective presentations of each of the summands.  Part (b) follows from the second row of \eqref{eq:33}, since $P_{\Lambda}^1$ is a projective $\bar\Lambda$-module by Lemma~\ref{lem:proj}(a).  Moreover, the same diagram implies that if $F(M)\cong M$ then $P_{\bar\Lambda}=0$ and $\Omega_\Lambda M \cong P_{\Lambda}^1$.  Hence $\textup{pd}_{\Lambda} M \leq 1$.  Conversely, if $\textup{pd}_{\Lambda} M \leq 1$ then $\Omega_\Lambda M \cong P_{\Lambda}^1$ and $\Omega^2_{\Lambda}M = 0$.  Thus, $\ker\,\pi_M \cong \Omega^{-1}_{\bar\Lambda} \Omega^2_{\Lambda} M=0$, which shows that $F(M)\cong M$, which completes the proof of part (c).   

To show (d), first suppose that $F(M) \cong N_1 \oplus N_2$ for some nonzero $N_1, N_2 \in \mod\,\bar\Lambda$ that have no summands of $\projinj\,\bar\Lambda$.  Then the minimal projective resolution of $F(M)$ from the second row of \eqref{eq:33} can also be written as follows, where $P_{\Lambda}^1 \cong  P_{\Lambda}^1(N_1)\oplus P_{\Lambda}^1(N_2)$ and $P_{\Lambda}^0 \cong  P_{\Lambda}^0(N_1)\oplus P_{\Lambda}^0(N_2)$.  Note that $f_1, f_2$ are nonzero since $N_1, N_2 \in \mod\,\bar\Lambda$ have no projective-injective summands. 

\[\xymatrix@C=30pt{
0\ar[r]& P_{\Lambda}^1(N_1)\oplus P_{\Lambda}^1(N_2)\ar[r]^-{\begin{bsmallmatrix}f_1&0\\ 0 & f_2 \\ \star & \star \end{bsmallmatrix}} & P_{\Lambda}^0(N_1)\oplus P_{\Lambda}^0(N_2)\oplus P_{\bar\Lambda} \ar[r]& N_1\oplus N_2 \ar[r]& 0}
\]

Moreover, by diagram \eqref{eq:33}, $M$ is the cokernel of $f=\begin{bsmallmatrix}f_1&0\\ 0 & f_2 \end{bsmallmatrix}: P_{\Lambda}^1 \to P_{\Lambda}^0$, where $f$ is a projective presentation of $M$.  This implies that $M \cong \coker\,f_1 \oplus \coker\,f_2$, and that the summands $\coker\,f_1, \coker\,f_2$ are nonzero.  This shows that if $M$ is decomposable then so is $F(M)$.  Conversely, part (a) implies that if $M$ is decomposable then so is $F(M)$, which completes the proof of (d). 
\end{proof}

\begin{lemma}\label{lem:inj}
The map $F: \textup{mod}\,\Lambda \to \textup{mod}\,\bar\Lambda$ is injective.
\end{lemma}

\begin{proof}
Let $M, N\in \text{mod}\,\Lambda$ such that $F(M)\cong F(N)$.  We need to show that $M\cong N$.  Since $F(M)\cong F(N)$ then they have isomorphic projective resolutions that are of the form given in Lemma~\ref{lem:33}.  Moreover, we obtain the following commutative diagram with projective resolutions of $F(M)$ and $F(N)$
\[
\xymatrix{0 \ar[r] & P^1_{\Lambda} \ar[r]^{\begin{bsmallmatrix}f\\g \end{bsmallmatrix}} \ar[d]^{\phi_1}& P^0_{\Lambda}\oplus P_{\bar\Lambda} \ar[r] \ar[d]^{\phi_0}& F(M) \ar[r] \ar[d]^{\cong}& 0 \\
0 \ar[r] & P^1_{\Lambda} \ar[r]^{\begin{bsmallmatrix}f'\\g' \end{bsmallmatrix}}& P^0_{\Lambda}\oplus P_{\bar\Lambda} \ar[r] & F(N) \ar[r] & 0}
\]
where $\phi_0, \phi_1$ are isomorphisms.  Here, $P_{\bar\Lambda}\in \projinj\bar\Lambda$ and $P^1_{\Lambda}, P^0_{\Lambda}\in \proj\,\Lambda$.   Moreover, note that by construction in Lemma~\ref{lem:33} we have $M \cong \coker\, f$ and $N \cong \coker\,f'$. 

Note that the top of $P_{\bar\Lambda}$ is in $\Lambda^\bullet$, so $\text{Hom}_{\bar\Lambda}(P_{\bar\Lambda}, P^0_{\Lambda})=0$ and $\phi_0=\begin{bsmallmatrix}\phi_{00} & 0 \\ \phi_{10}& \phi_{11} \end{bsmallmatrix}$, where $\phi_{00}, \phi_{11}$ are isomorphisms.  Then by commutativity of the above diagram we obtain $\phi_{00} f = f' \phi_1$.  Hence, 
\[M \cong \coker\, f \cong \coker \, \phi_{00} f = \coker \, f' \phi_1 \cong \coker\, f' \cong N.\]
\end{proof}

Now we define a subcategory $\mathcal{H}(\bar\Lambda)$ of $\text{mod}\,\bar\Lambda$, which we will show equals the image of the map $F$. 

\begin{definition}
Let $\mathcal{H}(\bar\Lambda)$ denote the full subcategory of $\text{mod}\,\bar\Lambda$ consisting of all modules of projective dimension at most one which do not contain summands of $\projinj\,\bar\Lambda$ nor a summand of $\Omega^{-1}_{\bar\Lambda}(\proj\,\Lambda)$. 
\end{definition}

\begin{lemma}\label{lem:image}
The image of $F:\textup{mod}\,\Lambda\to \textup{mod}\,\bar\Lambda$ is contained in $\mathcal{H}(\bar\Lambda)$. 
\end{lemma}

\begin{proof}
Since the map $F$ respects direct sum decompositions by Proposition~\ref{prop:F}(d), it suffices to show that the lemma holds for an indecomposable module $M\in \text{mod}\,\Lambda$.  By Proposition~\ref{prop:F}(b) the module $F(M)$ has projective dimension at most one.  Moreover, by definition $F(M)$ does not contain any projective-injective summands, so it suffices to show that $F(M)$ is not a summand of $\Omega^{-1}_{\bar\Lambda}(\proj\,\Lambda)$.  

Let $P_\Lambda \in \proj\,\Lambda$, and consider the following sequence 
\[0\to P_{\Lambda}\to I_{\bar\Lambda}\to \Omega^{-1}_{\bar\Lambda} P_{\Lambda}\to 0\]
where $I_{\bar\Lambda}$ is an injective envelope of $P_\Lambda$.  Since the socle of $P_{\Lambda}$ is in $\text{mod}\,\Lambda$, then by Lemma~\ref{lem:proj}(b) its injective envelope $I_{\bar\Lambda}$ is projective-injective.  In particular, the sequence above is also a minimal projective resolution of $\Omega^{-1}_{\bar\Lambda} P_{\Lambda}$ in $\text{mod}\,\bar\Lambda$.  If $F(M)\cong \Omega^{-1}_{\bar\Lambda}P_{\Lambda}$  then by Lemma~\ref{lem:33} the projective cover of $M$ in $\text{mod}\,\Lambda$ is zero.  Thus $M=0$ and $\Omega^{-1}_{\bar\Lambda}P_{\Lambda}=0$.  This shows the desired claim. 
\end{proof}

\begin{lemma}\label{lem:surj}
The map $F$ is surjective onto $\mathcal{H}(\bar\Lambda)$. 
\end{lemma}

\begin{proof}
Let $X$ be some nonzero module in $\mathcal{H}(\bar\Lambda)$.  Since $\pd_{\bar\Lambda} X\leq 1$, let 
\[0\to P^1_{\Lambda}\to P^0_{\Lambda}\oplus P_{\bar\Lambda}\to X \to 0\]
be a minimal projective resolution of $X$ in $\text{mod}\,\bar\Lambda$, where $P_{\bar\Lambda}$ is a maximal projective-injective summand of the projective cover of $X$.  Then $P^0_{\Lambda}\in\proj\,\Lambda$ and moreover $P^1_{\Lambda}\in\proj\,\Lambda$ because it cannot contain any summands of $\projinj\,\bar\Lambda$, as otherwise the resolution would not be minimal.  

If $P^1_{\Lambda}=0$, then $X$ is projective but not projective-injective, since $X\in  \mathcal{H}(\bar\Lambda)$.  Thus, $X\in \proj\,\Lambda$, so $F(X)=X$ by Proposition~\ref{prop:F}(c).  In particular, $X$ is in the image of $F$.

If $P^0_{\Lambda}=0$, then $P_{\bar\Lambda}\not=0$ and it is a minimal injective cover of $P^1_{\Lambda}$.   Thus, $X\cong \Omega^{-1}_{\bar\Lambda} P^1_{\Lambda}$, which contradicts $X\in  \mathcal{H}(\bar\Lambda)$.

Hence, suppose that $P^1_{\Lambda}, P^0_{\Lambda}$ are both nonzero.  Consider the map $f: P^1_{\Lambda}\to P^0_{\Lambda}$ coming from the sequence above.  Let $M = \coker\,f$.  We want to show that $F(M)\cong X$. 

Now consider the diagram 

\begin{equation}\label{eq:2}
\xymatrix{
0\ar[r] & P^1_{\Lambda}\ar[r]^-{\begin{bsmallmatrix}f\\g'\end{bsmallmatrix}}\ar@{=}[d] & P^0_{\Lambda}\oplus I(\ker\,f)\ar[r]\ar@{..>}[d]^-{\phi_0} & F(M)\oplus \Omega^{-1}_{\bar\Lambda} P \ar[r]\ar@{..>}[d]^{\phi} & 0 \\
0\ar[r] & P^1_{\Lambda}\ar[r]^-{\begin{bsmallmatrix}f\\g\end{bsmallmatrix}} & P^0_{\Lambda}\oplus P_{\bar\Lambda}\ar[r] & X\ar[r]& 0
}
\end{equation}

where the bottom sequence is the minimal projective presentation of $X$ from before. 
The top sequence is constructed as follows.  First, we can decompose the map $f:  P^1_{\Lambda}\to P^0_{\Lambda}$ into $f_{min}\oplus f'$ where $f_{min}: P^1_{min} \to P^0_{min}$ is the minimal projective presentation of $M = \coker\,f$, $f': P\oplus P'' \xrightarrow{[0, 1]} P''$ where $P, P''\in\proj\,\Lambda$.  Then $\coker\, f_{min} \cong \coker\, f = M$, and by Lemma~\ref{lem:33} there exists a short exact sequence for $F(M)$ 

\begin{equation}\label{eq:2a} \xymatrix@C=30pt{0\ar[r] & P^1_{min}  \ar[r]^-{\begin{bsmallmatrix}f_{min}\\ g_{min}\end{bsmallmatrix}} & P^0_{min}\oplus  I(\ker f_{min})\ar[r] & F(M) \ar[r] & 0} 
\end{equation}
where $I(\ker f_{min})$ is the minimal injective cover of $\ker f_{min}$ in $\mod\,\bar\Lambda$.   We also obtain the following short exact sequence 
\begin{equation}\label{eq:2b}
\xymatrix{0\ar[r] & P\ar[r] \oplus P''& I(\ker \,f')\oplus P'' \ar[r]& \Omega^{-1}_{\bar\Lambda}P\ar[r] & 0}
\end{equation}
where $I(\ker f')$ is the injective cover of $\ker\,f' = P$.  By taking a direct sum of the two short exact sequences \eqref{eq:2a} and \eqref{eq:2b}, we obtain the top row of the diagram \eqref{eq:2}, where $I(\ker\,f)$ is an injective cover of $\ker\,f$.  Now we want to show that there exists a map $\phi_0$ that makes the first square commute in \eqref{eq:2}, which then implies the existence of $\phi$.

Apply $\Hom_{\bar\Lambda}(-, P_{\bar\Lambda})$ to the top short exact sequence in diagram \eqref{eq:2} to obtain the following long exact sequence. 

\[\Hom_{\bar\Lambda}(P^0_{\Lambda}\oplus I(\ker\,f), P_{\bar\Lambda})\to \Hom_{\bar\Lambda}(P^1_{\Lambda}, P_{\bar\Lambda})\to \Ext^1_{\bar\Lambda}(F(M)\oplus \Omega^{-1}_{\bar\Lambda} P, P_{\bar\Lambda})\]

Note that the last term is zero, because $P_{\bar\Lambda}\in\projinj\,\bar\Lambda$, which implies the first map above is surjective.  Hence, the map $g \in \Hom_{\bar\Lambda}(P^1_{\Lambda}, P_{\bar\Lambda})$ factors through 
$\begin{bmatrix}f\\g'\end{bmatrix}$.   In particular there exists 
\[h=\begin{bmatrix}h_0 & h_1\end{bmatrix}: P^0_{\Lambda}\oplus I(\ker\,f)\to P_{\bar\Lambda}\] 
such that $g=h \circ \begin{bmatrix}f\\g'\end{bmatrix}$.   Now, we let $\phi_0=\begin{bmatrix}1 & 0 \\ h_0 & h_1\end{bmatrix}$, and observe that this makes the first square in diagram \eqref{eq:2} commute.  This completes the construction of this diagram.

Now, we want to show that $h_1$ is injective.  Let $\soc$ denote the socle of a module, and note that by construction we have 
\[\soc(\ker\,h_1)\subset \soc \,I(\ker\,f) = \soc (\im\, g')\]
since $I(\ker\,f)\in\projinj\,\bar\Lambda$ and $F(M)\oplus \Omega^{-1}_{\bar\Lambda}P$ has no summands of $I(\ker\, f)$.  Also, $\soc\,I(\ker\,f)=\soc(\ker\,f)$ since $I(\ker\,f)$ is the injective envelope of $\ker\,f$.  
Let $a$ be an element of $\soc(\ker\,h_1)$, then there exists $b\in \ker\,f\subset P^1_{\Lambda}$ such that $g'(b)=a$.  Thus, by commutativity of  \eqref{eq:2} we have 
\[\big( \phi_0 \circ \begin{bmatrix}f\\g'\end{bmatrix}\big) (b)=\phi_0 \begin{bmatrix}0\\a\end{bmatrix} = \begin{bmatrix}0\\ h_1(a)\end{bmatrix}=0 = \begin{bmatrix}f\\g \end{bmatrix} (b).\]
Since $\begin{bmatrix}f\\g \end{bmatrix}$ is injective, we conclude that $b=0$, so $a=0$.  Thus $\soc (\ker\, h_1)=0$ so $\ker\,h_1 =0$.  This shows that $h_1$ is injective.  In particular, by diagram \eqref{eq:2}  we obtain that the maps $\phi_0$ and $\phi$ are both injective.  

Now consider the short exact sequence 
\[0\to P^0_{\Lambda}\oplus I(\ker\,f)\xrightarrow{\phi_0} P^0_{\Lambda}\oplus P_{\bar\Lambda} \to \coker\,\phi_0\to 0.\]
In particular, this is a projective resolution of $\coker\,\phi_0$ in $\text{mod}\,\bar\Lambda$.  Because the map in the sequence on the summands $P^0_{\Lambda} \to  P^0_{\Lambda}$ is identity and since $I(\ker\,f)\in\projinj\,\bar\Lambda$, it follows that this sequence splits.  In particular, $\coker \, \phi_0$ is also a projective-injective.  
By the snake lemma we obtain that $\coker \, \phi_0\cong \coker \, \phi$.  Thus, the sequence 
\[0\to F(M)\oplus  \Omega^{-1}_{\bar\Lambda} P\to X \to \coker \, \phi \to 0\]
also splits, as it ends in an injective.  In particular, $X \cong F(M)\oplus  \Omega^{-1}_{\bar\Lambda} P \oplus \coker\,\phi$.  Since $X \in \mathcal{H}(\bar\Lambda)$ then it does not contain any summands of $\projinj\,\bar\Lambda$ nor summands of $\Omega^{-1}_{\bar\Lambda} \Lambda$.  This shows the desired claim that $X \cong F(M)$. 
\end{proof}

By combining the above results, we immediately obtain the following theorem. 

\begin{thm}\label{thm:bij}
The map $F: \textup{mod}\,\Lambda \to \mathcal{H}(\bar\Lambda)$ is a bijection. 
\end{thm}

\begin{proof}
By Lemma~\ref{lem:image} the map $F$ is well-defined.  Moreover, $F$ is a bijection by Lemma~\ref{lem:surj} and Lemma~\ref{lem:inj}.
\end{proof}

\begin{example}\label{ex:3}
If $\Lambda$ is the path algebra of $\xymatrix{1\ar[r]^\alpha & 2\ar[r]^\beta & 3}$ modulo the relation $\alpha\beta=0$, then $\bar\Lambda$ is the path algebra of the following quiver with relations $\alpha\beta=\alpha^\bullet\beta^\bullet = 0$ and $\gamma\alpha = \beta^\bullet\delta$. 
\[\xymatrix{
1^\bullet\ar[r]^{\alpha^\bullet} & 2^\bullet\ar[r]^{\beta^\bullet} \ar[dl]_\gamma& 3^\bullet \ar[dl]_\delta\\
1\ar[r]^\alpha & 2\ar[r]^\beta & 3
}\]
Observe that the projective $\bar\Lambda$-modules are 
\[\begin{smallmatrix}1\\2\end{smallmatrix}\oplus \begin{smallmatrix}2\\3\end{smallmatrix}\oplus \begin{smallmatrix}3\end{smallmatrix}\oplus \begin{smallmatrix}1^\bullet\\2^\bullet\\1\,\,\end{smallmatrix}\oplus \begin{smallmatrix}2^\bullet\\ 1\,3^\bullet\\ 2\,\,\end{smallmatrix}\oplus \begin{smallmatrix}3^\bullet\\2\,\,\\3\,\,\end{smallmatrix}\]
where the first three are also projective $\Lambda$-modules and the last three are the projective-injective $\bar\Lambda$-modules. 

The image of the simple module $S(1) \in \mod\,\Lambda$ under $F$ can be computed as follows.  The projective presentation of $S(1)$ is 
\[P_{\Lambda}(2)\to P_{\Lambda}(1)\to S(1)\to 0.\]
Then in order to make the map on the left injective we add $P_{\bar\Lambda}(3^\bullet)$ and obtain 
\[P_{\Lambda}(2)\to P_{\Lambda}(1)\oplus P_{\bar\Lambda}(3^\bullet)\to {\begin{smallmatrix} 3^\bullet \,\,1\\2\end{smallmatrix}}\to 0.\]
Thus $F(S(1))={\begin{smallmatrix} 3^\bullet \,\,1\\2\end{smallmatrix}}$.   Moreover, we can visualize the map $F$ on the level of Auslander-Reiten quivers as shown below.  Here the image of $F$ corresponds to the circled entries $\mathcal{H}(\bar\Lambda)$.  The boxed entries are the projective-injective $\bar\Lambda$-modules while the entries circled by dashes correspond to summands of $\Omega^{-1}_{\bar\Lambda} (\proj\,\Lambda)$.  Together all the highlighted $\bar\Lambda$-modules are precisely the modules of projective dimension at most 1.

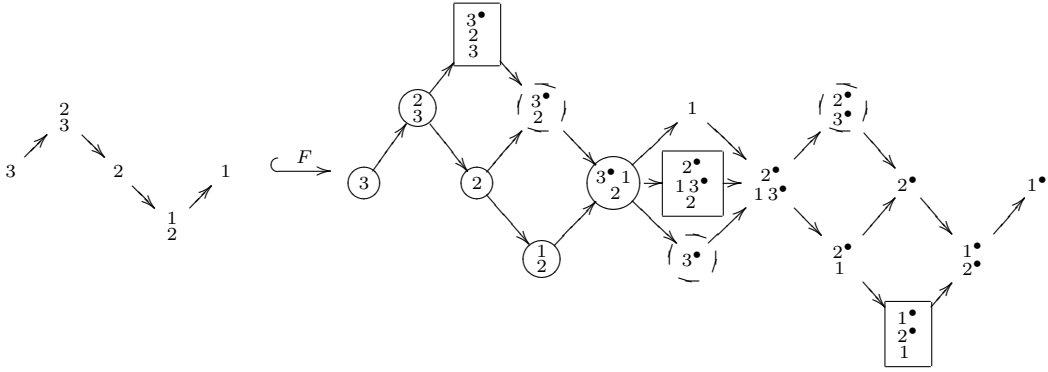
\[
\xymatrix@C=7pt@R=7pt{\\ \\
&{\begin{smallmatrix}2\\3\end{smallmatrix}}\ar[dr] \\
{\begin{smallmatrix}3 \end{smallmatrix}}\ar[ur] && {\begin{smallmatrix}2 \end{smallmatrix}}\ar[dr] && {\begin{smallmatrix}1\end{smallmatrix}} & \ar@{^{(}->}[rr]^{F} &&\\
&&&{\begin{smallmatrix}1\\2\end{smallmatrix}}\ar[ur]
}
\xymatrix@C=7pt@R=7pt{
&& *+[F]{\begin{smallmatrix}3^\bullet\\2\,\,\\3\,\,\end{smallmatrix}} \ar[dr]\\
&*+[o][F]{\begin{smallmatrix}2\\3\end{smallmatrix}}\ar[dr]\ar[ur] 
&&*+[F-o]{\begin{smallmatrix}3^\bullet\\2\,\,\end{smallmatrix}}\ar[dr] 
&& {\begin{smallmatrix}1\end{smallmatrix}}\ar[dr]&&
*+[F-o]{\begin{smallmatrix}2^\bullet\\3^\bullet \end{smallmatrix}}\ar[dr]\\
*+[o][F]{\begin{smallmatrix}3\end{smallmatrix}}\ar[ur] 
&& *+[o][F]{\begin{smallmatrix}2\end{smallmatrix}}\ar[ur]\ar[dr]
&& *+[o][F]{\begin{smallmatrix}3^\bullet \,1\\\,2\end{smallmatrix}}\ar[ur]\ar[dr]\ar[r] 
& *+[F]{\begin{smallmatrix} 2^\bullet\\1\,3^\bullet\\2\,\end{smallmatrix}}\ar[r] 
& {\begin{smallmatrix}2^\bullet\\1\,3^\bullet\end{smallmatrix}}\ar[ur]\ar[dr] 
&& {\begin{smallmatrix}2^\bullet \end{smallmatrix}} \ar[dr]
&& {\begin{smallmatrix}1^\bullet \end{smallmatrix}}\\
&&& *+[o][F]{\begin{smallmatrix}1\\2\end{smallmatrix}}\ar[ur] 
&& *+[F-o]{\begin{smallmatrix}3^\bullet\end{smallmatrix}}\ar[ur]
&& {\begin{smallmatrix}2^\bullet \\1\,\, \end{smallmatrix}}\ar[ur]\ar[dr]
&& {\begin{smallmatrix}1^\bullet\\2^\bullet\end{smallmatrix}}\ar[ur]\\
&&&&&&&&*+[F]{\begin{smallmatrix}1^\bullet\\2^\bullet \\ 1\,\,\end{smallmatrix}}\ar[ur]
}
\]
\end{example}

\section{$\tau$-tilting as classical tilting}

In this section we show that the map $F$ can be extended to a map $\bar F$ that yields a bijection between the support $\tau$-tilting pairs of $\Lambda$ and tilting modules of $\bar\Lambda$.  In this way $\tau$-tilting theory can be thought of as a special case of tilting theory.  

\subsection{The map $F$ and rigidity}
We begin with a number of preliminary lemmas about how the map $F$ preserves rigidity. 

\begin{lemma}
If $F(M)\in \textup{mod}\,\bar\Lambda$ is rigid then $M\in \textup{mod}\,\Lambda$ is rigid. 
\end{lemma}

\begin{proof}
Let $F(M)\in\text{mod}\,\bar\Lambda$ be rigid, and consider the following short exact sequence constructed in diagram~\eqref{eq:33}.  

\begin{equation}\label{sq:1}
0\to \ker\,\pi_M \to F(M)\xrightarrow{\pi} M \to 0
\end{equation}

Apply $\text{Hom}_{\bar\Lambda}(F(M), - )$ to get 
\[ \Ext^1_{\bar\Lambda}( F(M), F(M)) \to \Ext^1_{\bar\Lambda}( F(M), M) \to \Ext^2_{\bar\Lambda}( F(M), \ker\,\pi_M).\]
Note that the first term is zero since $F(M)$ is rigid, and the last term is zero by Proposition~\ref{prop:F}(b).  This implies that 
$\Ext^1_{\bar\Lambda}( F(M), M) =0$. 

Now apply $\text{Hom}_{\bar\Lambda}(-, M )$ to the sequence \eqref{sq:1} to obtain 
\[ \text{Hom}_{\bar\Lambda}(\ker\,\pi_M, M )\to \text{Ext}^1_{\bar\Lambda}(M, M )\to \text{Ext}^1_{\bar\Lambda}(F(M), M ).\]
The last term is zero by above, and the first term is zero because $\tp(\ker\,\pi_M)\in \text{mod}\,\Lambda^\bullet$ by diagram \eqref{eq:33}.   Therefore, $\text{Ext}^1_{\bar\Lambda}(M, M )=0$ which implies the desired conclusion. 
\end{proof}

The next three lemmas describe an important relationship between $\tau$ in $\text{mod}\,\Lambda$ and $\text{mod}\,\bar\Lambda$. 

\begin{lemma}\label{lem:tau}
Let $M\in \textup{mod}\,\Lambda$.  Then there exists a short exact sequence in $\textup{mod}\,\bar\Lambda$ 
\[ 0 \to \tau_\Lambda M \to \tau_{\bar\Lambda} M \to X^\bullet\to 0\]
such that $X^\bullet\in \textup{mod}\,\Lambda^\bullet$. 
\end{lemma}

\begin{proof}
Let 
\begin{equation}\label{eq:4}
P^1_{\Lambda} \to P^0_{\Lambda} \to M \to 0
\end{equation}
be a minimal projective presentation of $M$ in $\text{mod}\,\Lambda$, which is also a minimal projective presentation of $M$ in $\text{mod}\,\bar\Lambda$ by Lemma~\ref{lem:proj}(a).   Let $\nu_{\Lambda}, \nu_{\bar\Lambda}$ denote the Nakayama functors in the respective module categories.   Consider the following diagram, where the first row is obtained by applying $\nu_{\Lambda}$ to \eqref{eq:4} and the second row is obtained by applying $\nu_{\bar\Lambda}$ to the same sequence. 

\[\xymatrix{
0\ar[r] & \tau_{\Lambda} M \ar[r] \ar@{^{(}..>}[d]^i & \nu_{\Lambda} P^1_{\Lambda} \ar@{^{(}->}[d]^{i_1}\ar[r] & \nu_{\Lambda}P^0_{\Lambda} \ar@{^{(}->}[d]^{i_0}\\
0\ar[r] & \tau_{\bar\Lambda} M \ar[r] \ar@{->>}[d] & \nu_{\bar\Lambda} P^1_{\Lambda} \ar@{->>}[d]\ar[r] & \nu_{\bar\Lambda}P^0_{\Lambda} \\
&\coker\,i \ar@{^{(}->}[r] & P_{\Lambda^\bullet}^1
}\]

Recall that $\nu_{\bar\Lambda}$ of a projective $\Lambda$-module $P_{\Lambda}$ is a projective-injective $\bar\Lambda$-module $\nu_{\bar\Lambda}P_{\Lambda}$ by Lemma~\ref{lem:proj}(b).  Moreover, $\nu_{\Lambda} P_{\Lambda}$ embeds into $\nu_{\bar\Lambda}P_{\Lambda}$ and the cokernel is a projective $\Lambda^\bullet$-module $P_{\Lambda^\bullet}$.  
This implies that we have vertical injective maps $i_1, i_0$ as in the diagram, which make the rightmost square commute.   Then by the universal property of the kernel there exists a map $i: \tau_{\Lambda} M \to \tau_{\bar\Lambda} M$ that makes the first square commute.  In particular, $i$ is injective.   

Now, the snake lemma implies that $\coker\, i$ is a submodule of a projective $\Lambda^\bullet$-module $P_{\Lambda^\bullet}^1$, which implies that $\coker\,i \in \text{mod}\,\Lambda^\bullet$.  Setting $X^\bullet = \coker\,i$ we obtain the short exact sequence as in the statement of the lemma.
\end{proof}

\begin{lemma}\label{lem:tau1}
Let $M\in \textup{mod}\,\Lambda$. Then $\Hom_{\bar\Lambda}(M, \tau_{\bar\Lambda} M) \cong \Hom_{\Lambda}(M, \tau_{\Lambda} M)$.
\end{lemma}

\begin{proof}
Apply $\Hom_{\bar\Lambda}(M, - )$ to the short exact sequence in Lemma~\ref{lem:tau} to obtain a long exact sequence
\[0 \to \Hom_{\bar\Lambda}(M, \tau_{\Lambda}M)\to \Hom_{\bar\Lambda}(M, \tau_{\bar\Lambda}M)\to \Hom_{\bar\Lambda}(M, X^{\bullet} ).\]
Observe that the last term is zero because $M\in \text{mod}\,\Lambda$ and $X^\bullet \in \text{mod}\,\Lambda^\bullet$, which means that the two modules have disjoint support.  This shows the desired claim. 
\end{proof}

Similarly, we compare $\tau_{\bar\Lambda}$ of $M$ and $F(M)$.

\begin{lemma}\label{lem:tau2}
Let $M\in \textup{mod}\,\Lambda$. Then there exists a short sequence in $\textup{mod}\,\bar\Lambda$ 
\[ 0 \to \tau_{\bar\Lambda} F(M) \to \tau_{\bar\Lambda} M \to Y^\bullet\to 0\]
such that $Y^\bullet\in \textup{mod}\,\Lambda^\bullet$. 
\end{lemma}

\begin{proof}
Consider the minimal projective presentations of $F(M)$ and $M$ in $\text{mod}\,\bar\Lambda$ as in diagram \eqref{eq:33}, 
where $P_{\bar\Lambda}\in\projinj\,\bar\Lambda$.  

\[
\xymatrix{
0 \ar[r] & P^1_{\Lambda}\ar[r]^{\rho} \ar@{=}[d]& P^0_{\Lambda}\oplus P_{\bar\Lambda}\ar[r]\ar@{->>}[d]^{\begin{bsmallmatrix}1 & 0\end{bsmallmatrix}}& F(M) \ar[r] \ar@{->>}[d]& 0\\
&P^1_{\Lambda}\ar[r]^{\rho'} & P^0_{\Lambda}\ar[r] & M \ar[r] & 0
}
\]

Apply the Nakayama functor $\nu_{\bar\Lambda}$ to get the part of the following commutative diagram with the two exact rows.

\[
\xymatrix{&&&\ker\,h\ar@{^{(}..>}[dd]\ar@{^{(}->}[r]&\nu_{\bar\Lambda} P_{\bar\Lambda}\ar@{^{(}->}[d]\\
0 \ar[r] & \tau_{\bar\Lambda} F(M) \ar[r]^{i}\ar@{^{(}..>}[dd]^j &\nu_{\bar\Lambda} P^1_{\Lambda}\ar[rr]^(.3){\nu_{\bar\Lambda}\rho} \ar@{=}[dd]\ar@{->>}[dr]&& \nu_{\bar\Lambda}P^0_{\Lambda}\oplus \nu_{\bar\Lambda} P_{\bar\Lambda}\ar@{->>}[dd]^{\nu_{\bar\Lambda} \begin{bsmallmatrix}1 & 0\end{bsmallmatrix}}\\
&&&\coker \,i \ar@{^{(}->}[ur]\ar@{..>>}[dd]^(.3){h}\\
0\ar[r] & \tau_{\bar\Lambda} M\ar[r]^{i'} &\nu_{\bar\Lambda} P^1_{\Lambda}\ar[rr]^(.3){\nu_{\bar\Lambda}\rho'} \ar@{->>}[dr]&& \nu_{\bar\Lambda} P^0_{\Lambda}\\
&&&\coker \,i' \ar@{^{(}->}[ur]\\
}
\]

The map $j: \tau_{\bar\Lambda} F(M)\to \tau_{\bar\Lambda} M$ exists and moreover is injective by the universal property of the kernel of $\nu_{\bar\Lambda} \rho'$.  Now, observe that the maps $\nu_{\bar\Lambda}\rho, \nu_{\bar\Lambda}\rho'$ factor through the cokernels of $i, i'$ respectively.  Moreover, by the universal property of the cokernel of $i$ there exists a map $h: \coker\, i\to \coker\,i'$ such that the corresponding squares involving $h$ commute.  In particular, $h$ is surjective.  Finally, by the snake lemma $\coker\, j \cong \ker\, h$.  Hence it suffices to show that $\ker\, h\in \text{mod}\,\Lambda^\bullet$, and the lemma follows by setting $Y^\bullet = \ker\, h$.

By the universal property of $\ker\, (\nu_{\bar\Lambda} \begin{bsmallmatrix}1 & 0\end{bsmallmatrix})$, we observe that $\ker\, h$ embeds into $\ker\, (\nu_{\bar\Lambda} \begin{bsmallmatrix}1 & 0\end{bsmallmatrix})\cong \nu_{\bar\Lambda} P_{\bar\Lambda}$. Since $P_{\bar\Lambda}\in\projinj\bar\Lambda$, then by Lemma~\ref{lem:proj} parts (a) and (b), it is a projective module whose top is in $\text{mod}\,\Lambda^\bullet$.  Then the associated injective module $\nu_{\bar\Lambda} P_{\bar\Lambda}$ is an injective $\Lambda^\bullet$-module by Lemma~\ref{lem:proj}(c), so in particular $\nu_{\bar\Lambda} P_{\bar\Lambda}\in \text{mod}\,\Lambda^\bullet$.  This implies the desired conclusion that $\ker\, h\in \text{mod}\,\Lambda^\bullet$.
\end{proof}

The next two results show that $F$ induces a bijection between $\tau$-rigid $\Lambda$-modules and rigid $\bar\Lambda$-modules. 

\begin{prop}\label{prop:rigid}
If $F(M)$ is rigid in $\textup{mod}\,\bar\Lambda$ then $M$ is $\tau_{\Lambda}$-rigid. 
\end{prop}

\begin{proof}
Let $F(M)\in \text{mod}\,\bar\Lambda$ be rigid.  By Lemma~\ref{lem:tau1} it suffices to show that $\Hom_{\bar\Lambda}(M, \tau_{\bar\Lambda}M)=0$. 

Since the projective dimension of $F(M)$ is at most one by Proposition~\ref{prop:F}(b), then the Auslander Reiten formula implies the following. 
\[0 = \Ext^1_{\bar\Lambda}(F(M), F(M)) \cong D\Hom_{\bar\Lambda}(F(M), \tau_{\bar\Lambda} F(M))\]

Now, consider the sequence in the statement of Lemma~\ref{lem:tau2}
\[ 0 \to \tau_{\bar\Lambda} F(M) \xrightarrow{j} \tau_{\bar\Lambda} M \xrightarrow{h} Y^\bullet\to 0\]
Let $f\in \Hom_{\bar\Lambda}(M, \tau_{\bar\Lambda}M)$.  We want to show that $f=0$. Since $Y^\bullet\in \text{mod}\,\Lambda^\bullet$ and $M\in \text{mod}\,\Lambda$, then $hf=0$.  Then $f$ factors through the kernel of $h$, so in particular there exists a map $q: M \to \tau_{\bar\Lambda} F(M)$ such that $jq=f$.  
On the other hand, we have the projection $\pi_M: F(M)\to M$.  The composition $q \pi_M\in \Hom_{\bar\Lambda}(F(M), \tau_{\bar\Lambda} F(M))$ which is zero by above.  Then $q=0$ because $\pi_M$ is surjective, which implies that $f=jq=0$.  This shows the desired claim that $\Hom_{\bar\Lambda}(M, \tau_{\bar\Lambda}M)=0$. 
\end{proof}

We now show the converse of the previous statement. 

\begin{prop}\label{prop:rigid2}
If $M$ is $\tau_\Lambda$-rigid then $F(M)$ is rigid in $\textup{mod}\,\bar\Lambda$. 
\end{prop}

\begin{proof}
Let $M\in\text{mod}\,\Lambda$ be $\tau_{\Lambda}$-rigid, and apply $\Hom_{\bar\Lambda}(M, -)$ to the short exact sequence in Lemma~\ref{lem:tau2} to obtain the following. 
\[0\to \Hom_{\bar\Lambda}(M, \tau_{\bar\Lambda}F(M))\to \Hom_{\bar\Lambda}(M, \tau_{\bar\Lambda} M)\]
By Lemma~\ref{lem:tau1} the last term is isomorphic to $\Hom_{\bar\Lambda}(M, \tau_{\Lambda} M)$ which is zero by assumption.   This implies that 
\[0=\Hom_{\bar\Lambda}(M, \tau_{\bar\Lambda}F(M))\cong \Ext^1_{\bar\Lambda}(F(M), M)\]
where the isomorphism comes from applying the Auslander Reiten formula and noting that $\pd_{\bar\Lambda} F(M)\leq 1$ by Proposition~\ref{prop:F}(b).

Now consider the following short exact sequence 
\[0\to \Omega^2_{\Lambda} M \to P_{\bar\Lambda} \to \ker\,\pi \to 0\]
which comes from diagram \eqref{eq:33}.  Note that here $P_{\bar\Lambda}\in\projinj\,\bar\Lambda$.  Applying $\Hom_{\bar\Lambda}(F(M), - )$ to this sequence yields the following. 
\[\Ext^1_{\bar\Lambda}(F(M), P_{\bar\Lambda}) \to \Ext^1_{\bar\Lambda}(F(M), \ker\,\pi)\to \Ext^2_{\bar\Lambda}(F(M), \Omega^2_{\Lambda} M)  \]
The first term is zero because $P_{\bar\Lambda}$ is projective-injective, and the last term is zero because $\pd_{\bar\Lambda} F(M)\leq 1$.   Thus, we obtain that $\Ext^1_{\bar\Lambda}(F(M), \ker\,\pi)=0$. 

Lastly consider the short exact sequence 
\[0 \to \ker\,\pi \to F(M)\to M \to 0\]
and apply $\Hom_{\bar\Lambda}(F(M), - )$.  Thus, we obtain 
\[\Ext^1_{\bar\Lambda}(F(M), \ker\,\pi) \to \Ext^1_{\bar\Lambda}(F(M), F(M))\to \Ext^1_{\bar\Lambda}(F(M),M).\]
The first and the last terms vanish by the computations above, thus we obtain the desired result that $\Ext^1_{\bar\Lambda}(F(M), F(M))=0$.
\end{proof}

We combine the above results to obtain the following bijection. 

\begin{thm}\label{thm:rigid_bij}
The map $F$ restricts to a bijection between $\tau_{\Lambda}$-rigid modules in $\textup{mod}\,\Lambda$ and rigid $\bar\Lambda$-modules in $\mathcal{H}(\bar\Lambda)$. 
\end{thm}

\begin{proof}
The map $F$ is a bijection between $\text{mod}\,\Lambda$ and $\mathcal{H}(\bar\Lambda)$ by Theorem~\ref{thm:bij}.  Now the statement follows from Propositions~\ref{prop:rigid} and \ref{prop:rigid2}.
\end{proof}

\subsection{The map $\bar F$}
We extend $F$ to a map $\bar F$ between support $\tau$-tilting pairs in $\text{mod}\,\Lambda$  and tilting modules in $\text{mod}\,\bar\Lambda$.   We begin with the following definition. 

\begin{definition}\label{def:f}
Let $(M, P_{\Lambda})$ be a support $\tau$-tilting pair in $\text{mod}\,\Lambda$, and define 
\[\bar F(M, P_{\Lambda}) = F(M)\oplus \Omega^{-1}_{\bar\Lambda} P_{\Lambda}\oplus  Q_{\bar\Lambda}\] 
where $Q_{\bar\Lambda}$ is the direct sum of all indecomposable projective-injective $\bar\Lambda$-modules. 
\end{definition}

First we show that $\bar F$ maps projective $\Lambda$-modules in the second component of a $\tau$-rigid pair to rigid $\bar\Lambda$-modules. 

\begin{lemma}\label{lem:omega}
The module $\Omega^{-1}_{\bar\Lambda} \Lambda\in \textup{mod}\,\bar\Lambda$ is rigid and has projective dimension one. 
\end{lemma}

\begin{proof}
Let $I_{\bar\Lambda}$ be the injective envelope of $\Lambda$ in $\bar\Lambda$, then consider the following sequence in $\text{mod}\,\bar\Lambda$. 
\[0\to \Lambda \to I_{\bar\Lambda} \to \Omega^{-1}_{\bar\Lambda}\Lambda \to 0\]
Note that $I_{\bar\Lambda}$ is also projective by Lemma~\ref{lem:proj}(b), so in particular $\pd_{\bar\Lambda}  \Omega^{-1}_{\bar\Lambda}\Lambda =1$.  Next we show that $\Omega^{-1}_{\bar\Lambda} \Lambda$ is rigid.  Applying $\Hom_{\bar\Lambda}( -, \Omega^{-1}_{\bar\Lambda}\Lambda)$ to this sequence we obtain 
\[ \Hom_{\bar\Lambda}( I_{\bar\Lambda}, \Omega^{-1}_{\bar\Lambda}\Lambda)\to \Hom_{\bar\Lambda}( \Lambda, \Omega^{-1}_{\bar\Lambda}\Lambda)\to \Ext^1_{\bar\Lambda}( \Omega^{-1}_{\bar\Lambda}\Lambda, \Omega^{-1}_{\bar\Lambda}\Lambda)\to \Ext^1_{\bar\Lambda}( I_{\bar\Lambda}, \Omega^{-1}_{\bar\Lambda}\Lambda)\]
where the last term is zero because $I_{\bar\Lambda}\in\projinj\,\bar\Lambda$.  Also note that the first map is surjective, because $\Lambda\in \proj\,\bar\Lambda$ and $I_{\bar\Lambda}$ is a projective cover of $\Omega^{-1}_{\bar\Lambda} \Lambda$.  This implies that $\Ext^1_{\bar\Lambda}( \Omega^{-1}_{\bar\Lambda}\Lambda, \Omega^{-1}_{\bar\Lambda}\Lambda)=0$, so $\Omega^{-1}_{\bar\Lambda}\Lambda$ is rigid. 
\end{proof}

The next two lemmas show that the condition that $\Hom_\Lambda(P_\Lambda, M)=0$ is equivalent to the condition that $F(M)\oplus \Omega^{-1}_{\bar\Lambda} P_\Lambda$ is rigid.  

\begin{lemma}\label{lem:rigid4}
Let $M\in \textup{mod}\,\Lambda$ and $P_{\Lambda}\in\proj\,\Lambda$.  If $\Hom_{\Lambda}(P_\Lambda, M)=0$ then 
\begin{itemize}
\item[(a)] $\Ext^1_{\bar\Lambda}(F(M), \Omega^{-1}_{\bar\Lambda} P_{\Lambda})=0$ and 
\item[(b)] $\Ext^1_{\bar\Lambda}(\Omega^{-1}_{\bar\Lambda} P_{\Lambda}, F(M))=0$.
\end{itemize}
\end{lemma}

\begin{proof}
Part (a) follows since $\Ext^1_{\bar\Lambda}(F(M), \Omega^{-1}_{\bar\Lambda} P_{\Lambda})\cong \Ext^2_{\bar\Lambda}(F(M), P_{\Lambda})$ which is zero since $\pd_{\bar\Lambda}F(M)\leq 1$. 

To show part (b) consider the following short exact sequence 
\begin{equation}\label{eq:5}
0\to P_{\Lambda}\xrightarrow{i} I_{\bar\Lambda}\to \Omega^{-1}_{\bar\Lambda} P_{\Lambda}\to 0
\end{equation}
in $\text{mod}\,\bar\Lambda$ where  $I_{\bar\Lambda}$ is the injective cover of $P_{\Lambda}$.  Note that $I_{\bar\Lambda}$ is a projective-injective.  Applying $\Hom_{\bar\Lambda}(-, M)$ to this sequence we obtain 
\[\Hom_{\bar\Lambda}(P_{\Lambda}, M)\to \Ext^1_{\bar\Lambda}(\Omega^{-1}_{\bar\Lambda} P_{\Lambda}, M)\to \Ext^1_{\bar\Lambda}(I_{\bar\Lambda}, M).\]
Observe that the first term is zero by assumption and the last term is zero because $I_{\bar\Lambda} \in \projinj\,\bar\Lambda$.  Hence, we conclude that $\Ext^1_{\bar\Lambda}(\Omega^{-1}_{\bar\Lambda} P_{\Lambda}, M)=0$. 

Next consider the short exact sequence 
\[0 \to \ker\,\pi \to F(M)\to M \to 0\]
and apply $\Hom_{\bar\Lambda}(\Omega^{-1}_{\bar\Lambda} P_{\Lambda}, - )$ to obtain the following. 
\[\Hom_{\bar\Lambda}(\Omega^{-1}_{\bar\Lambda} P_{\Lambda}, M)\to \Ext^1_{\bar\Lambda}(\Omega^{-1}_{\bar\Lambda} P_{\Lambda}, \ker\,\pi )\to \Ext^1_{\bar\Lambda}(\Omega^{-1}_{\bar\Lambda} P_{\Lambda}, F(M))\to \Ext^1_{\bar\Lambda}(\Omega^{-1}_{\bar\Lambda} P_{\Lambda}, M )\]
Note that the first term is zero because $\tp \,\Omega^{-1}_{\bar\Lambda} P_{\Lambda}\in \text{mod}\,\Lambda^\bullet$ and the last term is zero by the calculation above.  Hence we conclude that 
\[\Ext^1_{\bar\Lambda}(\Omega^{-1}_{\bar\Lambda} P_{\Lambda}, \ker\,\pi )\cong \Ext^1_{\bar\Lambda}(\Omega^{-1}_{\bar\Lambda} P_{\Lambda}, F(M))\]
so it remains to show that $\Ext^1_{\bar\Lambda}(\Omega^{-1}_{\bar\Lambda} P_{\Lambda}, \ker\,\pi )=0$.  

Apply $\Hom_{\bar\Lambda}(-, \ker\,\pi)$ to the sequence \eqref{eq:5} to obtain
\[\Hom_{\bar\Lambda}(I_{\bar\Lambda}, \ker\,\pi)\xrightarrow{i^*} \Hom_{\bar\Lambda}(P_\Lambda, \ker\,\pi)\to \Ext^1_{\bar\Lambda}(\Omega^{-1}_{\bar\Lambda} P_{\Lambda}, \ker\,\pi)\to \Ext^1_{\bar\Lambda}(I_{\bar\Lambda}, \ker\,\pi).\]
The last term is zero since $I_{\bar\Lambda} \in \projinj\,\bar\Lambda$.  Hence, to obtain the desired conclusion we need to show that the map $i^*$ is surjective.  Let $f\in  \Hom_{\bar\Lambda}(P_\Lambda, \ker\,\pi)$, and recall from diagram~\ref{eq:33} that there exists a short exact sequence 
\[0\to \Omega^2_{\Lambda} M \to P'_{\bar\Lambda}\xrightarrow{g} \ker\,\pi\to 0\]
where $P'_{\bar\Lambda}\in \projinj\,\bar\Lambda$.  In particular, $g$ is a projective cover of $\ker\,\pi$, so $f$ factors through $g$.  Hence, there exists a map $f': P_{\Lambda}\to P'_{\bar\Lambda}$ such that $f=gf'$.  Now the map $f'$ whose target is an injective module must also factor through $i$, the injective envelope of $P_\Lambda$.  Thus, $f'=f''i$ for some $f'': I_{\bar\Lambda} \to P'_{\bar\Lambda}$.   Thus we obtain $f=gf'=gf''i$, which implies that $f\in \im\, i^*$.  This yields the desired conclusion. 
\end{proof}

\begin{lemma}\label{lem:pm}
If $F(M)\oplus \Omega^{-1}_{\bar\Lambda} P_{\Lambda}$ is rigid in $\textup{mod}\,\bar\Lambda$ then $\Hom_{\Lambda}(P_\Lambda, M)=0$. 
\end{lemma}

\begin{proof}
Consider the short exact sequence \eqref{eq:5} where $I_{\bar\Lambda}\in \projinj\,\bar\Lambda$.  Applying $\Hom_{\bar\Lambda}(-,M)$ we obtain 
\[\Hom_{\bar\Lambda}(I_{\bar\Lambda},M)\to \Hom_{\bar\Lambda}(P_\Lambda,M)\to \Ext^1_{\bar\Lambda}(\Omega^{-1}_{\bar\Lambda} P_{\Lambda},M)\to \Ext^1_{\bar\Lambda}(I_{\bar\Lambda}, M)\]
where the first term is zero because $\tp\, I_{\bar\Lambda}\in \text{mod}\,\Lambda^\bullet$ and the last term is zero since $I_{\bar\Lambda}$ is a projective-injective.  Thus,  $\Hom_{\bar\Lambda}(P_\Lambda,M)\cong \Ext^1_{\bar\Lambda}(\Omega^{-1}_{\bar\Lambda} P_{\Lambda},M)$, and it remains to show that the latter term $\Ext^1_{\bar\Lambda}(\Omega^{-1}_{\bar\Lambda} P_{\Lambda},M)$ is zero. 

Now, apply $\Hom_{\bar\Lambda}(\Omega^{-1}_{\bar\Lambda} P_{\Lambda}, - )$ to the short exact sequence
\[0 \to \ker\,\pi \to F(M)\to M \to 0\]
to obtain 
\[\Ext^1_{\bar\Lambda}(\Omega^{-1}_{\bar\Lambda} P_{\Lambda}, F(M) )\to  \Ext^1_{\bar\Lambda}(\Omega^{-1}_{\bar\Lambda} P_{\Lambda}, M )\to \Ext^2_{\bar\Lambda}(\Omega^{-1}_{\bar\Lambda} P_{\Lambda}, \ker\,\pi ).\]
The first term is zero by assumption and the last term is zero because $\pd_{\bar\Lambda} \Omega^{-1}_{\bar\Lambda} P_{\Lambda}\leq 1$.  This implies that $\Ext^1_{\bar\Lambda}(\Omega^{-1}_{\bar\Lambda} P_{\Lambda}, M )=0$ and completes the proof of the lemma. 
\end{proof}

We are now ready to prove one of our main theorems. 

\begin{thm}\label{thm:f}
The map $\bar F: \stilt\,\Lambda \to \tilt\,\bar\Lambda$ is a bijection.  
\end{thm}

\begin{proof}
Let $(M, P_\Lambda)$ be a support $\tau$-tilting pair.  Then by definition $\Hom_{\Lambda}(M, \tau_{\Lambda} M)=0$ and $\Hom_\Lambda(P_\Lambda,M)=0$.   By definition of $\bar F$ we have 
\[\bar F(M, P_{\Lambda}) = F(M)\oplus \Omega^{-1}_{\bar\Lambda} P_{\Lambda}\oplus Q_{\bar\Lambda}.\]
Note that the projective dimension of each of the summands of $\bar F(M, P_{\Lambda})$ is at most one by Proposition~\ref{prop:F}(b) and Lemma~\ref{lem:omega}.

Now we want to show that $\bar F(M, P_{\Lambda})$ has no self-extensions, and since the last term is projective-injective it suffices to show that 
\[\Ext^1_{\bar\Lambda}(F(M), F(M))\oplus \Ext^1_{\bar\Lambda}(\Omega^{-1}_{\bar\Lambda} P_{\Lambda}, \Omega^{-1}_{\bar\Lambda} P_{\Lambda})\oplus \Ext^1_{\bar\Lambda}(F(M), \Omega^{-1}_{\bar\Lambda} P_{\Lambda})\oplus \Ext^1_{\bar\Lambda}(\Omega^{-1}_{\bar\Lambda} P_{\Lambda}, F(M))=0.\]
The first term is zero by Proposition~\ref{prop:rigid2}, the second term is zero by Lemma~\ref{lem:omega}, and the remaining two terms are zero by Lemma~\ref{lem:rigid4}.   This shows that $\bar F(M, P_{\Lambda})$ is partial tilting in $\bar\Lambda$.  

Lastly, we check that $ \left|\bar F(M, P_{\Lambda})\right|=\left| \bar \Lambda \right|.$  By Proposition~\ref{prop:F}(a) we have that $\left|F(M)\right|\geq \left| M\right|$, and since $P_{\Lambda}$ does not contain any summands that are injective $\bar\Lambda$-modules then $\left|  \Omega^{-1}_{\bar\Lambda} P_{\Lambda}\right| \geq \left| P_{\Lambda}\right|$.  Therefore $\left| F(M)\oplus \Omega^{-1}_{\bar\Lambda} P_{\Lambda} \right|\geq \left| M\oplus P_{\Lambda}\right|$.  Since $\left| M\oplus P_{\Lambda}\right| = \left| \Lambda \right|$, we obtain 

\[\left| \bar F(M, P_{\Lambda}) \right| = \left| F(M)\oplus \Omega^{-1}_{\bar\Lambda} P_{\Lambda} \right| + \left|  Q_{\bar\Lambda} \right|  \geq  \left| M\oplus P_{\Lambda}\right| + \left| \Lambda \right|=2 \left| \Lambda \right| = \left| \bar \Lambda \right|\]

Since we know that $\bar F(M, P_{\Lambda})$ is rigid and has projective dimension at most one, then we also have that $\left| \bar F(M, P_{\Lambda}) \right| \leq \left| \bar \Lambda \right|$.  This implies that $\left| \bar F(M, P_{\Lambda}) \right| = \left| \bar \Lambda \right|$.

Combining all of the above, we conclude that $\bar F(M, P_{\Lambda})$ is a tilting $\bar\Lambda$-module. 

Now, it remains to show that $\bar F$ is surjective.  Let $\bar M \in \bar\Lambda$ be a tilting module.  Note that $ Q_{\bar\Lambda}$ belongs to every tilting module as it has no nonzero extensions with any module.  Hence, since $\pd_{\bar\Lambda} \bar M \leq 1$, we can decompose $\bar M$ as follows
\[\bar M \cong M' \oplus \Omega^{-1}_{\bar \Lambda} P_{\Lambda} \oplus  Q_{\bar\Lambda}\]
where $M'\in \mathcal{H}(\bar\Lambda)$ and $P_{\Lambda}\in\proj\,\Lambda$.  Since $F$ is surjective onto $\mathcal{H}(\bar\Lambda)$, there exists $M\in \text{mod}\,\Lambda$ such that $F(M)\cong M'$.  Now we want to show that the pair $(M, P_{\Lambda})$ is $\tau_{\Lambda}$-rigid. Indeed, $\Hom_{\Lambda}(M, \tau_{\Lambda} M)=0$ by Theorem~\ref{thm:rigid_bij}, and $\text{Hom}_{\Lambda}(P_\Lambda, M)=0$ by Lemma~\ref{lem:pm}.  Thus, $\bar F(M, P_{\Lambda}) \cong \bar M$.  This shows that $\bar F$ is surjective and completes the proof of the theorem. 
\end{proof}

We will show in the next section that $\bar F$ is order preserving, see Corollary~\ref{cor:all}.  Indeed, the poset $\stilt\,\Lambda$ is isomorphic to the poset $\tilt\,\bar\Lambda$ under $\bar F$ and forms a Bongartz interval in  $\stilt\,\bar\Lambda$.

\begin{example}
Let $\Lambda$ be the path algebra given by a quiver as in Example~\ref{ex:3}.  Then $\stilt\,\Lambda$ is shown on the left, and under $\bar F$ it is mapped to $\tilt\,\bar\Lambda$ as shown on the right.  

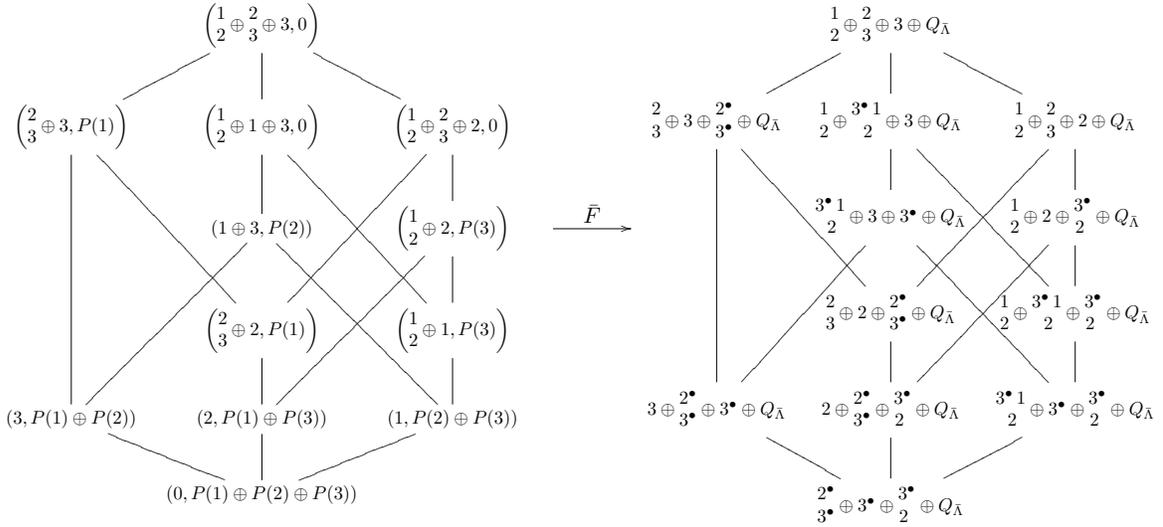
\[ \scalebox{.7}{\xymatrix@C=10pt@R=25pt{
&\left( {\begin{matrix}1\\2 \end{matrix}}\oplus {\begin{matrix} 2\\3 \end{matrix}} \oplus 3, 0\right)\ar@{-}[dr]\ar@{-}[dl] \ar@{-}[d]\\
 \left({\begin{matrix} 2\\3 \end{matrix}} \oplus 3, P(1) \right) \ar@{-}[ddd]\ar@{-}[ddr]
 & \left( {\begin{matrix}1\\2 \end{matrix}}\oplus 1 \oplus 3, 0\right) \ar@{-}[ddr] \ar@{-}[d]
 &  \left( {\begin{matrix}1\\2 \end{matrix}}\oplus {\begin{matrix} 2\\3 \end{matrix}} \oplus 2, 0\right) \ar@{-}[d] \ar@{-}[ddl]\\
&(1\oplus 3, P(2)) \ar@{-}[ddr]\ar@{-}[ddl]
&  \left( {\begin{matrix}1\\2 \end{matrix}}\oplus  2, P(3)\right) \ar@{-}[d] \ar@{-}[ddl] & \ar[rrr]^{\scalebox{1.2}{$\bar F$}}&&&\\
&\left({\begin{matrix} 2\\3 \end{matrix}} \oplus 2, P(1) \right) \ar@{-}[d]
& \left( {\begin{matrix}1\\2 \end{matrix}}\oplus  1, P(3)\right)\ar@{-}[d]&&\\
 (3, P(1)\oplus P(2)) \ar@{-}[dr]  & (2, P(1)\oplus P(3)) \ar@{-}[d] & (1, P(2)\oplus P(3)) \ar@{-}[dl] \\
&(0, P(1)\oplus P(2)\oplus P(3))
}}
\hspace{.0cm}\scalebox{.7}{\xymatrix@C=10pt@R=22pt{
& {\begin{matrix}1\\2 \end{matrix}}\oplus {\begin{matrix} 2\\3 \end{matrix}} \oplus 3\oplus Q_{\bar\Lambda} \ar@{-}[dr]\ar@{-}[dl] \ar@{-}[d]\\
{\begin{matrix} 2\\3 \end{matrix}} \oplus 3 \oplus {\begin{matrix} 2^
\bullet\\3^\bullet \end{matrix}} \oplus Q_{\bar\Lambda}  \ar@{-}[ddd]\ar@{-}[ddr]
 & {\begin{matrix}1\\2 \end{matrix}}\oplus  {\begin{matrix}3^
\bullet \,1\\\,2 \end{matrix}} \oplus 3  \oplus Q_{\bar\Lambda} \ar@{-}[ddr] \ar@{-}[d]
 &   {\begin{matrix}1\\2 \end{matrix}}\oplus {\begin{matrix} 2\\3 \end{matrix}} \oplus 2 \oplus Q_{\bar\Lambda} \ar@{-}[d] \ar@{-}[ddl]\\
& {\begin{matrix} 3^\bullet \,1\\\,2\end{matrix}} \oplus 3 \oplus 3^\bullet \oplus Q_{\bar\Lambda} \ar@{-}[ddr]\ar@{-}[ddl]
&  {\begin{matrix}1\\2 \end{matrix}}\oplus  2 \oplus {\begin{matrix}3^\bullet \\2\,\, \end{matrix}} \oplus Q_{\bar\Lambda} \ar@{-}[d] \ar@{-}[ddl]\\
&{\begin{matrix} 2\\3 \end{matrix}} \oplus 2\oplus {\begin{matrix}2^\bullet \\3^\bullet \end{matrix}}  \oplus Q_{\bar\Lambda} \ar@{-}[d]
& {\begin{matrix}1\\2 \end{matrix}}\oplus  {\begin{matrix} 3^\bullet \,1\\\,2\end{matrix}} \oplus {\begin{matrix}3^\bullet \\2\,\, \end{matrix}}  \oplus Q_{\bar\Lambda} \ar@{-}[d]\\
 3\oplus {\begin{matrix}2^\bullet \\3^\bullet \end{matrix}}\oplus 3^\bullet  \oplus Q_{\bar\Lambda} \ar@{-}[dr]  
 & 2\oplus  {\begin{matrix}2^\bullet \\3^\bullet \end{matrix}}\oplus {\begin{matrix}3^\bullet \\2\,\, \end{matrix}}  \oplus Q_{\bar\Lambda} \ar@{-}[d] &  {\begin{matrix} 3^\bullet \,1\\\,2\end{matrix}} \oplus 3^\bullet \oplus {\begin{matrix}3^\bullet \\2\,\, \end{matrix}}  \oplus Q_{\bar\Lambda} \ar@{-}[dl] \\
&{\begin{matrix}2^\bullet \\3^\bullet \end{matrix}}\oplus 3^\bullet \oplus {\begin{matrix}3^\bullet \\2\,\, \end{matrix}}  \oplus Q_{\bar\Lambda}
}}
\]
\end{example}

\section{Bongartz intervals in $\bar\Lambda$}

In this section we define another map $G: \mod\,\Lambda \to \mod\,\bar\Lambda$ that maps $\tau$-rigid $\Lambda$-modules to $\tau$-rigid $\bar\Lambda$-modules.  We extend it to an inclusion $\bar G: \stilt\,\Lambda\to \stilt\,\bar\Lambda$ on support $\tau$-tilting pairs and show that the image of $\bar G$ is a Bongartz interval in $\stilt\,\bar\Lambda$ isomorphic to $\stilt\,\Lambda$.   Moreover, every support $\tau$-tilting pair in the image of $\bar G$ is a maximal element of a different Bongartz interval in $\stilt\,\bar\Lambda$ which is also isomorphic to $\stilt\,\bar\Lambda$.  As an application we conclude that the product of posets $\stilt\,\Lambda \times \stilt\,\Lambda$ is a subposet of $\stilt\,\bar\Lambda$.  Moreover, it induces a similar embedding on the level of maximal green sequences.

Recall that $\bar\Lambda$ contains two isomorphic copies of $\Lambda$, which we denote by $\Lambda, \Lambda^\bullet$.  We begin by defining two maps $\mod\,\Lambda \to \mod\,\bar\Lambda$.

\begin{definition}
Define a functor $(-)^\bullet: \mod\,\Lambda \to \mod\,\Lambda^\bullet$  by mapping $\mod\,\Lambda$ isomorphically to $\mod\,\Lambda^\bullet$.
\end{definition}

Note that there is no ambiguity with this notation in the sense that the module $\Lambda^\bullet \in \mod\,\bar\Lambda$ is the same as the module obtained by applying the map $(-)^\bullet$ to $\Lambda$. 

By composing  $(-)^\bullet$ with the inclusion $\mod\,\Lambda^\bullet\to \mod\,\bar\Lambda$ we will consider $(-)^\bullet$ as a functor with image in $\bar\Lambda$. 

\begin{definition}\label{def:G}
Define a map $G: \mod\,\Lambda \to \mod\,\bar\Lambda$ on the objects as follows. 
First, given a projective $\Lambda$-module $P_{\Lambda}=\bigoplus_{i} P_{\Lambda}(i)^{m_i}$ let $G(P_{\Lambda})= \bigoplus_{i} P_{\bar\Lambda}(i^\bullet)^{m_i}$ be the associated projective $\bar\Lambda$-module at the dotted vertices.  
Now, given $M\in \mod \,\Lambda$ with projective presentation 
\[\xymatrix{P^1_{\Lambda} \ar[r]^{f} & P^0_{\Lambda}\ar[r] & M \ar[r] & 0}\]
define $G(M)$ to be the cokernel of the map $G(f)$ 
\[\xymatrix{G(P^1_{\Lambda}) \ar[r]^{G(f)} & G(P^0_{\Lambda})\ar[r] & G(M) \ar[r] & 0}.\]
Note that a map $f$ between projective $\Lambda$-modules is induced by paths in the quiver, so $G(f)$ is obtained by taking analogous paths between the dotted vertices in the subquiver of $\Lambda^\bullet$ viewed as a subalgebra of $\bar\Lambda$. 
\end{definition}

The following observations follow immediately from the definition of $G$. 

\begin{lemma}\label{lem:g0}
Let $M\in\mod\,\Lambda$.  Then
\begin{itemize}
\item[(a)] $M$ is indecomposable if and only if $G(M)$ is indecomposable, and 
\item[(b)] there is a short exact sequence $\xymatrix{0\ar[r]& \ker\,\rho_M  \ar[r]&G(M) \ar[r]^{\rho_M} & M^\bullet\ar[r] & 0}$ in $\mod\,\bar\Lambda$ with $\ker\,\rho_M\in\mod\,\Lambda$.
\end{itemize}
\end{lemma}

\begin{proof}
Part (a) follows because the map $f$ in Definition~\ref{def:G} decomposes as a direct sum of two nonzero maps if and only if the same holds for the map $G(f)$.  Part (b) can be shown as follows.  For projective modules $P_{\Lambda}$ there is a natural surjection $\rho: G(P_{\Lambda})\to (P_{\Lambda})^\bullet$, which then induces the surjection $\rho_M: G(M) \to M^\bullet$.  In particular, we have the following commutative diagram with exact rows. 
\[
\xymatrix{G(P^1_{\Lambda}) \ar[r]^{G(f)} \ar@{->>}[d]^{\rho_1}& G(P^0_{\Lambda})\ar[r] \ar@{->>}[d]^{\rho_0}& G(M) \ar[r] \ar@{->>}[d]^{\rho_M}& 0\\
(P^1_{\Lambda})^\bullet \ar[r]^{f^\bullet} & (P^0_{\Lambda})^\bullet\ar[r] & M^\bullet \ar[r] & 0}
\]

Then the snake lemma implies that there is a surjection $\ker\,\rho_0\to \ker\,\rho_M$.  By Lemma~\ref{lem:proj}(b), we have that $\ker\,\rho_0$ is an injective $\Lambda$-module, which implies that $\ker\,\rho_M\in\mod\,\Lambda$. 
\end{proof}

We also have the following relationship between $\tau$ and the map $G$. 

\begin{lemma}\label{lem:g1}
Let $M\in\mod\,\Lambda$.  Then $\tau_{\bar\Lambda} G(M) \cong (\tau_{\Lambda} M)^\bullet$. 
\end{lemma}

\begin{proof}
Let $M\in\mod\,\Lambda$ with projective presentation 
\[\xymatrix{P^1_{\Lambda} \ar[r]^{f} & P^0_{\Lambda}\ar[r] & M \ar[r] & 0}.\] 
Applying the Nakayama functor $\nu_{\Lambda}$ yields an exact sequence 
\begin{equation}\label{eq:g1}
\xymatrix{0\ar[r] & \tau_{\Lambda} M \ar[r] & \nu_{\Lambda} P^1_{\Lambda} \ar[r]^{\nu_{\Lambda} f} & \nu _{\Lambda}P^0_{\Lambda} }.
\end{equation}

By definition $G(P_{\Lambda}(i))\cong P_{\bar\Lambda}(i^\bullet)$ for every vertex $i$ of $\Lambda$, so applying the Nakayama functor over $\bar\Lambda$, we obtain the following sequence of isomorphisms 
\[\nu_{\bar\Lambda} G(P_{\Lambda}(i))\cong\nu_{\bar\Lambda}P_{\bar\Lambda}(i^\bullet)  \cong I_{\bar\Lambda}(i^\bullet)\cong (I_{\Lambda}(i))^{\bullet}.\]

In particular, applying the map $G$ and then $\nu_{\bar\Lambda}$ to the projective presentation of $M$ yields 
\begin{equation}\label{eq:g2}
\xymatrix{0\ar[r] & \tau_{\bar\Lambda} G(M) \ar[r] & (I^1_{\Lambda})^\bullet \ar[r]^{\nu_{\bar\Lambda} G(f)} & (I^0_{\Lambda})^\bullet }.
\end{equation}
Moreover, it is easy to see that the map $\nu_{\bar\Lambda} G(f)$ is the same as the map $(\nu_{\Lambda} f)^\bullet$. This shows that the sequence~\eqref{eq:g2} is obtained from ~\eqref{eq:g1} by applying the functor $(-)^\bullet$, and therefore we conclude that  $\tau_{\bar\Lambda} G(M) \cong (\tau_{\Lambda} M)^\bullet$. 
\end{proof}

Next we show that the map $G$ preserves $\tau$-rigidity. 

\begin{prop}\label{prop:g}
The pair $(M, P_{\Lambda})$ is $\tau$-rigid in $\mod\,\Lambda$ if and only if the pair $G(M, P_{\Lambda}):=(G(M), G(P_{\Lambda}))$ is $\tau$-rigid in $\mod\,\bar\Lambda$. 
\end{prop}

\begin{proof}
Given $M\in\mod\,\Lambda$ and $P_{\Lambda} \in \proj\,\Lambda$, 
consider the following sequence of isomorphisms 
\[\Hom_{\bar\Lambda}(G(M), \tau_{\bar\Lambda}G(M))\cong \Hom_{\bar\Lambda}(G(M), (\tau_{\Lambda}M)^\bullet)\cong \Hom_{\Lambda^\bullet}(M^{\bullet}, (\tau_{\Lambda}M)^\bullet)  \cong \Hom_{\Lambda}(M, \tau_{\Lambda}M),\]
where the first isomorphism follows from Lemma~\ref{lem:g1}.  The second isomorphism follows by applying the functor $\Hom_{\bar\Lambda}(-, (\tau_{\Lambda} M)^\bullet)$ to the short exact sequence in Lemma~\ref{lem:g0}(b) and noting that $\Hom_{\bar\Lambda}(\ker\,\rho_M, (\tau_{\Lambda} M)^\bullet)=0$.  This shows that $M$ is  $\tau$-rigid in $\mod\,\Lambda$ if and only if $G(M)$ is $\tau$-rigid in $\mod\,\bar\Lambda$. 

Moreover, we also have 
\[\Hom_{\bar\Lambda}(G(P_{\Lambda}), G(M))\cong \Hom_{\Lambda^\bullet}((P_{\Lambda})^\bullet, M^\bullet)\cong \Hom_{\Lambda}(P_{\Lambda}, M),\]
where the first isomorphism follows by applying $\Hom_{\bar\Lambda}(G(P_\Lambda), -)$ to the short exact sequence in Lemma~\ref{lem:g0}(b) and noting that $G(P_{\Lambda})$ is a projective module at the dotted vertices of $\bar\Lambda$.  This shows that $(M, P_{\Lambda})$ is a $\tau$-rigid pair in $\mod\,\Lambda$ if and only if $G(M, P_{\Lambda})$ is a $\tau$-rigid pair in $\mod\,\bar\Lambda$. 
\end{proof}

We extend the definition of $G$ to a map $\bar G$ on support $\tau$-tilting modules.  Define 
\begin{align*}
\bar G: \,\, & \stilt\,\Lambda \to \stilt\,\bar\Lambda  \\
& (M, P_{\Lambda})\mapsto (G(M)\oplus \Lambda, G(P_{\Lambda}))\\
\end{align*}
where $\Lambda$ is the direct sum of all indecomposable projective $\Lambda$-modules. The map $\bar G$ is indeed well-defined as shown below. 

\begin{thm}\label{thm:g}
The map $\bar G: \stilt\,\Lambda \to \stilt\,\bar\Lambda$ is an inclusion.  Moreover, $\bar G(M, P_{\Lambda})$ is the Bongartz completion of $G(M, P_{\Lambda})$ in $\mod\,\bar\Lambda$. 
\end{thm}

\begin{proof}
First, we show that the map $\bar G$ is well-defined.  Note that by Lemma~\ref{lem:g0}(b) the map $G$ is injective and its image is not a $\Lambda$-module apart from the zero module.  Moreover, by Lemma~\ref{lem:g0}(a) $G$ preserves the number of indecomposable summands.  Thus $\bar G(M, P_{\Lambda})$ has the desired number of indecomposable summands since
\[ \left|  G(M) \oplus G(P_{\Lambda})\oplus \Lambda \right| = \left| M\oplus P_{\Lambda}\right| +\left| \Lambda\right| =2\left| \Lambda\right| = \left| \bar\Lambda\right|.\]
Now, we show that  $\bar G(M, P_{\Lambda})$ is $\tau$-rigid.  Note that $\Lambda$ is a projective $\bar\Lambda$-module so $\tau_{\bar\Lambda}  (\Lambda_{\Lambda})=0$.  By Proposition~\ref{prop:g} we know that $(G(M), G(P_{\Lambda}))$ is $\tau$-rigid in $\mod\,\bar\Lambda$, so it suffices to show that adding $\Lambda$ as a direct summand to $G(M)$ preserves $\tau$-rigidity.  In particular, we need to show that 
\[ \Hom_{\bar\Lambda} (G(P_{\Lambda}), \Lambda) = 0=\Hom_{\bar\Lambda}(\Lambda, \tau_{\bar\Lambda}G(M)).\]
The first equality follows because $G(P_{\Lambda})$ is a projective at the dotted vertices of $\bar\Lambda$ and $\Lambda$ is not supported at the dotted vertices.  The second equality follows similarly after applying Lemma~\ref{lem:g1}.  This shows that $\bar G(M, P_{\Lambda})$ is $\tau$-rigid.  

Moreover, it is easy to see that $\bar G$ is an inclusion since $G$ is injective and its image does not contain nonzero summands of $\Lambda$. Lastly, $\bar G$ is the Bongartz completion of $G(M, P_{\Lambda})$, because $\Lambda$ is projective in $\mod\,\bar\Lambda$.   
\end{proof}

Next, we describe the image of $\bar G$ as the set of support $\tau$-tilting $\bar\Lambda$-modules containing $\Lambda$, that is, the Bongartz interval of $\Lambda$. 

\begin{thm}\label{cor:g}
The map $\bar{G}: \stilt\,\Lambda\to \stilt_{\Lambda}\,\bar\Lambda$ is an isomorphism of posets.
\end{thm}

\begin{proof}
By Theorem~\ref{thm:g}, the image $\bar{G}(\stilt\,\Lambda)$ is contained in the Bongartz interval  of $\Lambda$ in $\mod\,\bar\Lambda$.  To show the reverse inclusion, let $(N\oplus \Lambda, P_{\bar\Lambda})$ be a support $\tau$-tilting pair in $\mod\,\bar\Lambda$.  By Proposition~\ref{prop:g}, it suffices to show that $(N\oplus \Lambda, P_{\bar\Lambda})=\bar G(M, P_{\Lambda})$ for some $M\in\mod\,\Lambda$ and $P_{\Lambda}\in \proj\,\Lambda$.  First, observe that since $\Hom_{\bar\Lambda}( P_{\bar\Lambda}, \Lambda)=0$, then $P_{\bar\Lambda}$ is a projective at the dotted vertices of $\bar\Lambda$, and hence $P_{\bar\Lambda}$ is in the image of $G$.  Therefore, it suffices to show that $N$ is also in the image of $G$.  

 Let  
\[P^1_{\bar\Lambda}\xrightarrow{\bar f} P^0_{\bar\Lambda}\to N \to 0\]
be a projective presentation of $N$ in $\mod\,\bar\Lambda$.  Applying the Nakayama functor $\nu_{\bar\Lambda}$ we obtain an exact sequence
\[0\to \tau_{\bar\Lambda}N \to \nu_{\bar\Lambda} P^1_{\bar\Lambda}\to \nu_{\bar\Lambda} P^0_{\bar\Lambda}.\]
Since $N\oplus \Lambda$ is $\tau$-rigid, it follows that $\Hom_{\bar\Lambda}(\Lambda, \tau_{\bar\Lambda} N)=0$, so in particular $\tau_{\bar\Lambda} N$ is only supported at the dotted vertices of $\bar\Lambda$, so $\tau_{\bar\Lambda} N\in \mod\,\Lambda^\bullet$.  The sequence above then implies that $\nu P^1_{\bar\Lambda}, \nu P^0_{\bar\Lambda}$ are both injective modules at the dotted vertices, and hence $P^1_{\bar\Lambda}, P^0_{\bar\Lambda}$ are both projectives at the dotted vertices of $\bar\Lambda$.  Let $G(P^1_{\Lambda})=P^1_{\bar\Lambda}, G(P^0_{\Lambda})=P^0_{\bar\Lambda}$ for $P^1_{\Lambda}, P^0_{\Lambda}\in\proj\,\Lambda$.  Also, let $f: P^1_{\Lambda}\to P^0_{\Lambda}$ be the map such that $G(f)=\bar f$.   Define $M$ to be the cokernel of $f$.  Note that since $\bar f$ is a projective presentation of $N$ then $f$ is a projective presentation of $M$.  Then by definition of $G$, we have that $G(M)\cong N$.  This shows that $N$ is in the image of $G$ as desired. 
\end{proof}

\begin{example}\label{ex:4}
Let $\Lambda$ be as in Example~\ref{ex:2}.  Its $\tau$-tilting poset is shown on the left while its image under $\bar G$ is shown on the right. 
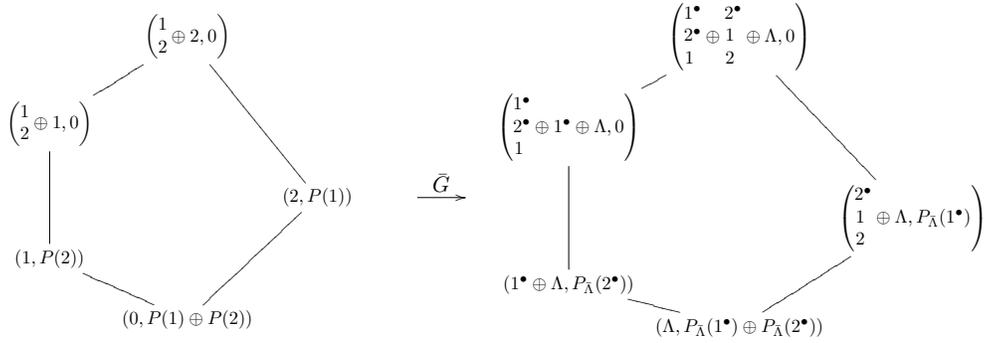
\[ \scalebox{.7}{\xymatrix@C=10pt@R=17pt{
&\left( {\begin{matrix}1\\2 \end{matrix}}\oplus 2, 0\right)\ar@{-}[ddr]\ar@{-}[dl] \\
 \left({\begin{matrix} 1\\2 \end{matrix}} \oplus 1, 0 \right) \ar@{-}[dd]\\
 && \left(2, P(1)\right) \ar@{-}[ddl] && \ar[rr]^{\scalebox{1.2}{$\bar G$}}&&\\
(1, P(2)) \ar@{-}[dr]  \\
&(0,P(1)\oplus P(2))
}}
\hspace*{.2cm}\scalebox{.7}{\xymatrix@C=2pt@R=7pt{
&\left( {\begin{matrix}1^\bullet\\2^\bullet\\1\,\, \end{matrix}}\oplus {\begin{matrix}2^\bullet\\1\,\,\\2\,\, \end{matrix}} \oplus \Lambda, 0\right)\ar@{-}[ddr]\ar@{-}[dl] \\
 \left({\begin{matrix}1^\bullet\\2^\bullet\\1\,\, \end{matrix}} \oplus 1^\bullet \oplus \Lambda, 0 \right) \ar@{-}[dd]\\
 && \left( {\begin{matrix}2^\bullet\\1\,\,\\2\,\, \end{matrix}} \oplus \Lambda, P_{\bar\Lambda}(1^\bullet)\right) \ar@{-}[ddl] \\
(1^\bullet \oplus \Lambda , P_{\bar\Lambda}(2^\bullet)) \ar@{-}[dr]  \\
&(\Lambda, P_{\bar\Lambda}(1^\bullet)\oplus P_{\bar\Lambda}(2^\bullet))
}}
\]
\end{example}

Recall that by Theorem~\ref{thm:f} the image of the map $\bar F$ is isomorphic to  $\stilt\,\Lambda$.  Theorem~\ref{cor:g} above shows similarly that the image of  $\bar G$ is the Bongartz interval of  the projective module $\Lambda$ which is also isomorphic to  $\stilt\,\Lambda$.  Note that these two images overlap only at the top element $(\bar\Lambda,0)$. Now we apply $\tau$-tilting reduction to conclude that $\stilt\,\Lambda \times \stilt\,\Lambda$ is a subposet of $\stilt\,\bar\Lambda$.  In particular, $\stilt\,\bar\Lambda$ contains many Bongartz intervals isomorphic to $\stilt\,\Lambda$.

\begin{prop}\label{prop:FM}
Let $(M, P_{\Lambda})$ be a support $\tau$-tilting pair in $\textup{mod}\,\Lambda$. Then there is a poset isomorphism 
\[\bar F_M: \stilt\,\Lambda \to \stilt_{G(M,P)}\bar\Lambda\] 
such that if $\bar F_M (M', P_{\Lambda}')=(N, P_{\bar\Lambda}')$ then $M' \cong fN$ is the torsion free part of $N$ with respect to the torsion pair $(\Fac (G(M)), G(M)^\perp)$ in $\textup{mod}\,\bar\Lambda$. 
\end{prop}

\begin{proof}
Let $(M, P_{\Lambda})$ be a support $\tau$-tilting pair in $\mod\,\Lambda$.  By Theorem~\ref{thm:g} $\bar G(M, P_{\Lambda})$ is a support $\tau$-tilting pair in $\mod\,\bar\Lambda$ and is the Bongartz completion of $G(M, P_{\Lambda})$.  We claim that the  Bongartz interval of $G(M, P_{\Lambda})$ is isomorphic to $\stilt\,\Lambda$.  By $\tau$-tilting reduction, see Theorem~\ref{thm:red}, the Bongartz interval of $G(M, P_{\Lambda})$ is isomorphic to the support $\tau$-tilting poset of the algebra 
\[C_{G(M, P_{\Lambda})}=\End_{\bar\Lambda}(G(M)\oplus \Lambda)/\left<e_{G(M)}\right>\]
 where $e_{G(M)}$ is the idempotent in the endomorphism algebra corresponding to $G(M)$.  It is then easy to see that since the top of $G(M)$ is at the dotted vertices then 
\[C_{G(M, P_{\Lambda})}=\End_{\bar\Lambda}(G(M)\oplus \Lambda)/\left<e_{G(M)}\right>\cong \End _{\Lambda}(\Lambda)\cong \Lambda.\]
This shows our claim that the  Bongartz interval of $G(M, P_{\Lambda})$ is isomorphic to $\stilt\,\Lambda$.  Now, define $\bar F_M$ to be the inverse of the map as from Theorem~\ref{thm:red} that sends support $\tau$-tilting pairs in $\mod\,\Lambda$ to their associated elements in the Bongartz interval of $G(M, P_{\Lambda})$.  

It remains to show that if $\bar F_M (M', P_{\Lambda}')=(N, P_{\bar\Lambda}')$ then $M' \cong fN$ is the torsion free part of $N$ with respect to the torsion pair $(\Fac (G(M)), G(M)^\perp)$ in $\textup{mod}\,\bar\Lambda$.  Given $(N, P_{\bar\Lambda}') \in \stilt_{G(M,P)}\bar\Lambda$, Theorem~\ref{thm:red} implies that $M'\cong \Hom_{\bar \Lambda}(G(M)\oplus \Lambda, fN)$ where $fN$ is defined by the torsion pair $(\Fac (G(M)), G(M)^\perp)$.  Since $N\in \Fac(G(M)\oplus \Lambda)$ it follows that $fN\in \Fac \,\Lambda$.  In particular $fN\in \mod\,\Lambda$ and \[M'\cong \Hom_{\bar \Lambda}(G(M)\oplus \Lambda, fN)\cong \Hom_{\bar \Lambda}(\Lambda, fN)\cong fN.\]
\end{proof}

Now we take the maps $\bar F_M$ for all possible support $\tau$-tilting pairs $(M, P_{\Lambda})$ to obtain an inclusion of posets $\stilt\,\Lambda \times \stilt\,\Lambda  \to \stilt\,\bar\Lambda$. 

\begin{corollary}\label{cor:cross}
There is an inclusion of posets 
\begin{align*}
\varphi: \stilt\,\Lambda \times \stilt\,\Lambda & \to \stilt\,\bar\Lambda\\
 ((M, P_{\Lambda}), (M', P_{\Lambda}'))& \mapsto \bar F_M (M', P_{\Lambda}').
 \end{align*}
\end{corollary}

\begin{proof}
By Proposition~\ref{prop:FM} the map $\varphi$ is well-defined.  First, we show that $\varphi$ is an inclusion.  By above $\bar F_M$ is an inclusion for a fixed $M$, therefore it suffices to show that the images of $\bar F_M$ and $\bar F_{M'}$ are disjoint for $M\not\cong M'$.  By construction, the image of $\bar F_M$ is the Bongartz interval of $G(M, P_{\Lambda})$, and the image of $\bar F_{M'}$ is the Bongartz interval of $G(M', P'_{\Lambda})$.  Any support $\tau$-tilting pair in the intersection of both of these intervals would contain $G(M\oplus M',   P_{\Lambda}\oplus P'_{\Lambda})$.  In particular, the pair $G(M\oplus M',   P_{\Lambda}\oplus P'_{\Lambda})$ is $\tau$-rigid, so Proposition~\ref{prop:g} implies that $(M\oplus M',   P_{\Lambda}\oplus P'_{\Lambda})$ is $\tau$-rigid in $\mod\,\Lambda$.  However, both $(M, P_{\Lambda}), (M', P'_{\Lambda})$ are support $\tau$-tilting pairs in $\mod\,\Lambda$.  Therefore, $M\cong M', P_{\Lambda}\cong P'_{\Lambda}$.   This shows that the images of $\bar F_M$ and $\bar F_{M'}$ are disjoint, and completes the proof that $\varphi$ is an inclusion. 

Now, we want to show that $\varphi$ respects the partial order.  Let $(M_1, P_1)\leq  (M_3, P_3)$ and $(M_2, P_2)\leq  (M_4, P_4)$ be support $\tau$-tilting pairs in $\mod\,\Lambda$. Then we want to show that $\bar F_{M_1}(M_2, P_2) \leq \bar F_{M_3}(M_4, P_4)$.   By definition of $\varphi$ and Proposition~\ref{prop:FM}, we have $\bar F_{M_1}(M_2, P_2) = (G(M_1)\oplus N_2, P_{\bar\Lambda})$ for some $N_2\in \mod\,\bar\Lambda$ such that there exists a short exact sequence in $\mod\,\bar\Lambda$ 
\[0\to tN_2\to N_2\to M_2\to 0\]
with $tN_2 \in \Fac\, G(M_1)$.  Similarly, $\bar F_{M_3}(M_4, P_4) = (G(M_3)\oplus N_4, P'_{\bar\Lambda})$ for some $N_4\in \mod\,\bar\Lambda$ such that $M_4$ is a quotient of $N_4$.  
Since $(M_2, P_2)\leq  (M_4, P_4)$, then $M_2\in \Fac\, M_4\subset \Fac \, N_4$, and since $(M_1, P_1)\leq  (M_3, P_3)$ then $G(M_1)\in \Fac \,G(M_3)$ by Theorem~\ref{cor:g}.  Thus, $tN_2, M_2 \in \Fac (G(M_3)\oplus N_4)$.  Since $\Fac (G(M_3)\oplus N_4)$ is a torsion class in $\mod\,\bar\Lambda$ it is closed under extensions, hence $N_2 \in \Fac (G(M_3)\oplus N_4)$.   Therefore, we conclude that $\Fac (G(M_1)\oplus N_2)\subset \Fac (G(M_3)\oplus N_4)$, which shows the desired conclusion $\bar F_{M_1}(M_2, P_2) \leq \bar F_{M_3}(M_4, P_4)$.
\end{proof}

Next, we relate the map $\bar F_M$ to the map $\bar F$ from Definition~\ref{def:f}.   Note that the image of $\bar F_{\Lambda}$ is the Bongartz interval of $G(\Lambda)$.  By definition $G(\Lambda)$ is the direct sum of the projective-injective $\bar\Lambda$-modules, which we denote by $Q_{\bar\Lambda}$, and we show below that the Bongartz interval of   $Q_{\bar\Lambda}$ is precisely the tilting poset $\tilt\,\bar\Lambda$.   In particular, we show that the maps $\bar F_{\Lambda}$ and $\bar F$ are equivalent.

\begin{thm}\label{prop:fg}
The poset inclusions $\bar F, \bar F_{\Lambda}: \stilt\,\Lambda \to \stilt\,\bar\Lambda$ are equal. 
In particular, $\bar F: \stilt\,\Lambda \to \tilt\,\bar\Lambda$ is a poset isomorphism. 
\end{thm}

\begin{proof}
By Proposition~\ref{prop:FM} the map $\bar F_{\Lambda}^{-1}: \stilt_{Q_{\bar\Lambda}}\bar\Lambda \to \stilt\, \Lambda$ is a poset isomorphism.  Moreover, a $\tau$-tilting $\bar\Lambda$-module $N$ is mapped to $\bar F_{\Lambda}^{-1}(N)= fN$ where $fN$ is the torsion free part of $N$ with respect to the torsion pair $(\Fac \,Q_{\bar\Lambda},   (Q_{\bar\Lambda})^{\perp}) =(\Fac \,Q_{\bar\Lambda},   \mod\,\Lambda)$ in $\mod\,\bar\Lambda$.  

Then the canonical sequence for $N$ with respect to this torsion pair is
\[0\to tN\to N \to fN \to 0\]
where $tN\in \Fac \,Q_{\bar\Lambda}, fN\in \mod\,\Lambda$.  
Note that $fN\in\mod\,\Lambda$. 

On the other hand the map $\bar F: \stilt \,\Lambda \to \tilt\,\bar\Lambda$ is a bijection by Theorem~\ref{thm:f}.  Moreover, $\tilt\,\bar\Lambda \subseteq \stilt_{Q_{\bar\Lambda}}\bar\Lambda$. 

Next we show that $\bar F ^{-1}$ also corresponds to taking the torsion free part with respect to the torsion pair $(\Fac \,Q_{\bar\Lambda},   \mod\,\Lambda)$.  Let $(M, P_{\Lambda})\in\stilt\,\Lambda$.  Then by definition we have 
 $\bar F (M, P_{\Lambda}) = F(M)\oplus \Omega^{-1}_{\bar\Lambda}P_{\Lambda}\oplus Q_{\bar\Lambda}$.  Note that $\Omega^{-1}_{\bar\Lambda}P_{\Lambda}\oplus Q_{\bar\Lambda} \in \Fac\,Q_{\bar\Lambda}$.  Now, the canonical sequence for $F(M)$ is precisely the last column of diagram \eqref{eq:33}, since $M\in \mod\,\Lambda$ and $\ker\,\pi_M$ is a quotient of a projective-injective $\bar\Lambda$-module.  This shows that $f(\bar F(M))\cong M$.   Therefore, the two maps $\bar F^{-1}, \bar F^{-1}_{\Lambda}: \tilt\,\bar\Lambda \to \stilt\,\Lambda$ agree.  
 
 Since $\bar F: \stilt\,\Lambda\to \tilt\,\bar\Lambda\subseteq \stilt_{Q_{\bar\Lambda}}\bar \Lambda$ is a bijection and $\bar F_{\Lambda}: \stilt\,\Lambda\to \stilt_{Q_{\bar\Lambda}}\bar \Lambda$ is a bijection such that their inverses agree on  $\tilt\,\bar\Lambda$, then we conclude that $\bar F$ and $\bar F_{\Lambda}$ are actually equal.  In particular, since $\bar F_{\Lambda}$ is a poset isomorphism then so is $\bar F$. 
 \end{proof}

In particular, the next statement  follows immediately from Theorem~\ref{prop:fg}.
\begin{corollary}\label{cor:all}
There is an isomorphism of posets 
\[\stilt\,\Lambda \cong  \stilt_{\projinj\,\bar\Lambda}\,\bar\Lambda \cong \tilt\,\bar\Lambda.\]
\end{corollary}
\begin{proof}
The isomorphism between the first two terms is given by $\bar F_{\Lambda}$ from Proposition~\ref{prop:FM}, and the isomorphism between the first and the last term is given by $\bar F$ from Theorem~\ref{prop:fg}.
\end{proof}

\begin{example}\label{ex:5}
	Let $\Lambda$ be as in Example~\ref{ex:2}.  Recall that the support $\tau$-tilting poset of $\Lambda$ is shown in Example~\ref{ex:4} and has the form of a pentagon.  The support $\tau$-tilting poset of $\bar\Lambda$ is shown in Figure~\ref{fig:big}.  The five highlighted pentagons denote the image of $\varphi$ inside $\stilt\,\bar\Lambda$, where each pentagon is a Bongartz interval of $G(M, P_{\Lambda})$ for a support $\tau$-tilting pair  $(M, P_{\Lambda})$ in $\mod\,\Lambda$.  Moreover, the maximal element in each highlighted region corresponds to the image of $\bar G$, which is also computed in Example~\ref{ex:4}.
Lastly, the top-most highlighted pentagon of $\stilt\,\bar\Lambda$ is $\tilt\,\bar\Lambda$ corresponding to the image of $\bar F$, or equivalently, the image of $\bar F_{\Lambda}$.   
	\end{example}
	
	\begin{figure}
\newcommand{\ob}{\smash{1^\bullet}}
\newcommand{\tb}{\smash{2^\bullet}}
\newcommand{\miniscule}{\fontsize{4}{5} \selectfont}
\begin{tikzpicture}[xscale=.7,yscale=.69]
	\newcommand{\ho}{0}
	\newcommand{\vo}{0}
	\draw [fill=red!30, color=red!30, rounded corners] (-2+\ho,-0.5+\vo) -- (2+\ho,-0.5+\vo) -- (4+\ho,1.5+\vo) -- (4+\ho,4.5+\vo) -- (2+\ho,6.5+\vo) -- (-2+\ho,6.5+\vo) -- (-4+\ho,3+\vo) -- cycle;

	\renewcommand{\ho}{7}
	\renewcommand{\vo}{8}
	\draw [fill=red!30, color=red!30, rounded corners] (-2+\ho,-0.5+\vo) -- (2+\ho,-0.5+\vo) -- (4+\ho,1.5+\vo) -- (4+\ho,4.5+\vo) -- (2+\ho,6.5+\vo) -- (-2+\ho,6.5+\vo) -- (-4+\ho,3+\vo) -- cycle;

	\renewcommand{\ho}{7}
	\renewcommand{\vo}{16}
	\draw [fill=red!30, color=red!30, rounded corners] (-2+\ho,-0.5+\vo) -- (2+\ho,-0.5+\vo) -- (4+\ho,1.5+\vo) -- (4+\ho,4.5+\vo) -- (2+\ho,6.5+\vo) -- (-2+\ho,6.5+\vo) -- (-4+\ho,3+\vo) -- cycle;

	\renewcommand{\ho}{0}
	\renewcommand{\vo}{24}
	\draw [fill=red!30, color=red!30, rounded corners] (-2+\ho,-0.5+\vo) -- (2+\ho,-0.5+\vo) -- (4+\ho,1.5+\vo) -- (4+\ho,4.5+\vo) -- (2+\ho,6.5+\vo) -- (-2+\ho,6.5+\vo) -- (-4+\ho,3+\vo) -- cycle;

	\renewcommand{\ho}{-5}
	\renewcommand{\vo}{17}
	\draw [fill=red!30, color=red!30, rounded corners] (-2+\ho,-0.5+\vo) -- (2+\ho,-0.5+\vo) -- (4+\ho,1.5+\vo) -- (4+\ho,4.5+\vo) -- (2+\ho,6.5+\vo) -- (-2+\ho,6.5+\vo) -- (-4+\ho,3+\vo) -- cycle;

	\node	(A)
		at (0,30)
		{\miniscule $\left(\begin{matrix}\ob\\\tb\\1\;\;\end{matrix}\oplus\begin{matrix}\tb\\1\;\;\\2\;\;\end{matrix}\oplus\begin{matrix}1\\2\end{matrix}\oplus2,0\right)$};
	\node	(B)
		at (-2,27)
		{\miniscule $\left(\begin{matrix}\ob\\\tb\\1\;\;\end{matrix}\oplus\begin{matrix}\tb\\1\;\;\\2\;\;\end{matrix}\oplus\tb\oplus2,0\right)$};
	\node	(C)
		at (2,28)
		{\miniscule $\left(\begin{matrix}\ob\\\tb\\1\;\;\end{matrix}\oplus\begin{matrix}\tb\\1\;\;\\2\;\;\end{matrix}\oplus\begin{matrix}1\\2\end{matrix}\oplus1,0\right)$};
	\node	(D)
		at (2,26)
		{\miniscule $\left(\begin{matrix}\ob\\\tb\\1\;\;\end{matrix}\oplus\begin{matrix}\tb\\1\;\;\\2\;\;\end{matrix}\oplus\begin{matrix}\tb\\1\;\;\end{matrix}\oplus1,0\right)$};
	\node	(E)
		at (0,24)
		{\miniscule $\left(\begin{matrix}\ob\\\tb\\1\;\;\end{matrix}\oplus\begin{matrix}\tb\\1\;\;\\2\;\;\end{matrix}\oplus\begin{matrix}\tb\\1\;\;\end{matrix}\oplus\tb,0\right)$};

	\node	(H)
		at (-5,23)
		{\miniscule $\left(\begin{matrix}\tb\\1\;\;\\2\;\;\end{matrix}\oplus\begin{matrix}1\\2\end{matrix}\oplus2,P(\ob)\right)$};
	\node	(K)
		at (-7,20)
		{\miniscule $\left(\begin{matrix}\tb\\1\;\;\\2\;\;\end{matrix}\oplus\tb\oplus2,P(\ob)\right)$};
	\node	(M)
		at (-3,21)
		{\miniscule $\left(\begin{matrix}\tb\\1\;\;\\2\;\;\end{matrix}\oplus\begin{matrix}1\\2\end{matrix}\oplus1,P(\ob)\right)$};
	\node	(Q)
		at (-3,19)
		{\miniscule $\left(\begin{matrix}\tb\\1\;\;\\2\;\;\end{matrix}\oplus\begin{matrix}\tb\\1\;\;\end{matrix}\oplus1,P(\ob)\right)$};
	\node	(R)
		at (-5,17)
		{\miniscule $\left(\begin{matrix}\tb\\1\;\;\\2\;\;\end{matrix}\oplus\begin{matrix}\tb\\1\;\;\end{matrix}\oplus\tb,P(\ob)\right)$};

	\node	(L)
		at (7,22)
		{\miniscule $\left(\begin{matrix}\ob\\\tb\\1\;\;\end{matrix}\oplus\ob\oplus\begin{matrix}1\\2\end{matrix}\oplus2,0\right)$};
	\node	(P)
		at (9,20)
		{\miniscule $\left(\begin{matrix}\ob\\\tb\\1\;\;\end{matrix}\oplus\ob\oplus\begin{matrix}1\\2\end{matrix}\oplus1,0\right)$};
	\node	(O)
		at (5,19)
		{\miniscule $\left(\begin{matrix}\ob\\\tb\\1\;\;\end{matrix}\oplus\ob\oplus\begin{matrix}\ob\\\tb\end{matrix}\oplus2,0\right)$};
	\node	(U)
		at (7,16)
		{\miniscule $\left(\begin{matrix}\ob\\\tb\\1\;\;\end{matrix}\oplus\ob\oplus\begin{matrix}\ob\\\tb\end{matrix},P(2)\right)$};
	\node	(V)
		at (9,18)
		{\miniscule $\left(\begin{matrix}\ob\\\tb\\1\;\;\end{matrix}\oplus\ob\oplus1,P(2)\right)$};

	\node	(X)
		at (7,14)
		{\miniscule $\left(\ob\oplus\begin{matrix}1\\2\end{matrix}\oplus2,P(\tb)\right)$};
	\node	(AC)
		at (5,11)
		{\miniscule $\left(\ob\oplus2,P(\tb)\oplus P(1)\right)$};
	\node	(AE)
		at (7,8)
		{\miniscule $\left(\ob,P(\tb)\oplus P(1)\oplus P(2)\right)$};
	\node	(AD)
		at (9,10)
		{\miniscule $\left(\ob\oplus1,P(\tb)\oplus P(2)\right)$};
	\node	(Z)
		at (9,12)
		{\miniscule $\left(\ob\oplus\begin{matrix}1\\2\end{matrix}\oplus 1,P(\tb)\right)$};

	\node	(AG)
		at (0,6)
		{\miniscule $\left(\begin{matrix}1\\2\end{matrix}\oplus2,P(\ob)\oplus P(\tb)\right)$};
	\node	(AH)
		at (-2,3)
		{\miniscule $\left(2,P(\ob)\oplus P(\tb)\oplus P(1)\right)$};
	\node	(AK)
		at (0,0)
		{\miniscule $\left(0,P(\ob)\oplus P(\tb)\oplus P(1)\oplus P(2)\right)$};
	\node	(AJ)
		at (2,2)
		{\miniscule $\left(1,P(\ob)\oplus P(\tb)\oplus P(2)\right)$};
	\node	(AI)
		at (2,4)
		{\miniscule $\left(\begin{matrix}1\\2\end{matrix},P(\ob)\oplus P(\tb)\oplus P(1)\right)$};

	\node[color=gray]	(AF)
		at (-7,6)
		{\miniscule $\left(\tb,P(\ob)\oplus P(1)\oplus P(2)\right)$};
	\node[color=gray]	(AA)
		at (-5.4,14)
		{\miniscule $\left(\tb\oplus 2,P(\ob)\oplus P(1)\right)$};
	\node[color=gray]	(AB)
		at (-3,8)
		{\miniscule $\left(\tb\oplus\begin{matrix}\tb\\1\;\;\end{matrix},P(\ob)\oplus P(2)\right)$};
	\node[color=gray]	(W)
		at (-4,11)
		{\miniscule $\left(1\oplus\begin{matrix}\tb\\1\;\;\end{matrix},P(\ob)\oplus P(2)\right)$};
	\node[color=gray]	(Y)
		at (2,9)
		{\miniscule $\left(\begin{matrix}\ob\\\tb\end{matrix}\oplus\ob,P(1)\oplus P(2)\right)$};
	\node[color=gray]	(T)
		at (2.5,14)
		{\miniscule $\left(\begin{matrix}\ob\\\tb\end{matrix}\oplus\tb,P(1)\oplus P(2)\right)$};
	\node[color=gray]	(S)
		at (-1.2,12.5)
		{\miniscule $\left(\begin{matrix}\ob\\\tb\end{matrix}\oplus\ob\oplus 2,P(1)\right)$};
	\node[color=gray]	(N)
		at (2,17)
		{\miniscule $\left(\begin{matrix}\ob\\\tb\\1\;\;\end{matrix}\oplus\begin{matrix}\ob\\\tb\end{matrix}\oplus\tb,P(2)\right)$};
	\node[color=gray]	(J)
		at (-1.4,16)
		{\miniscule $\left(\begin{matrix}\ob\\\tb\end{matrix}\oplus\tb\oplus 2,P(1)\right)$};
	\node[color=gray]	(I)
		at (2,20)
		{\miniscule $\left(\begin{matrix}\ob\\\tb\\1\;\;\end{matrix}\oplus\begin{matrix}\tb\\1\;\;\end{matrix}\oplus\tb,P(2)\right)$};
	\node[color=gray]	(G)
		at (8,24)
		{\miniscule $\left(\begin{matrix}\ob\\\tb\\1\;\;\end{matrix}\oplus\begin{matrix}\tb\\1\;\;\end{matrix}\oplus1,P(2)\right)$};
	\node[color=gray]	(F)
		at (4,24)
		{\miniscule $\left(\begin{matrix}\ob\\\tb\\1\;\;\end{matrix}\oplus\begin{matrix}\ob\\\tb\end{matrix}\oplus\tb\oplus 2,0\right)$};

	\draw[color=gray]
		(AK) to [bend left] (AF)
		(AJ) to [out=146, in=225] (W)
		(AH) to [bend left] (AA)
		(AE) to (Y)
		(AC) to (S)
		(Y) to [out=70, in=220] (U) 
		(U) to (N)
		(V) to [out=30, in=-40] (G)
		(O) to (F)
		(S) to [out=65, in=-90] (O)
		(AA) to [bend left] (K)
		(AB) to [out=54, in=-69] (R)
		(W) to [out=83, in=-92] (Q)
		(I) to (E)
		(F) to [out=160, in=-30] (B)
		(G) to (D)
		;

	\draw[color=gray]
		(AF) to (AA)
		(AF) to (AB)
		(AF) to [out=55, in=-135](T)
		(AA) to (J)
		(AB) to (W)
		(AB) to [out=20, in=-130] (I)
		(Y) to (S)
		(Y) to (T)
		(W) to [out=86, in=197] (G) 
		(T) to (J)
		(T) to (N)
		(S) to (J)
		(N) to (I)
		(N) to [out=50.5, in=-90] (F)
		(J) to [out=80, in=-135] (F)
		(I) to (G)
		;

	\draw[color=black]
		(A) to [bend right] (H)
		(B) to [bend right=32] (K)
		(C) to [bend right=20] (M)
		(D) to [bend left=35] (Q)
		(E) to [bend left=40] (R)
		(A) to [out=-10, in=110] (L)
		(C) to [out=-7, in=43] (P)
		(L) to [out=-70, in=157] (X)
		(P) to [out=-40, in=50] (Z)
		(V) to [out=-50, in=35] (AD)
		(X) to [out=215, in=106] (AG)
		(H) to [out=-19, in=91] (AG)
		(M) to [out=207, in=175] (AI)
		(Z) to [out=210, in=80] (AI)
		(AD) to [out=-75, in=30] (AJ)
		(AK) to [out=15, in=-90] (AE)
		(AC) to [out=190, in=96] (AH)
		;

	\draw[color=red, thick] (A) -- (B) -- (E) -- (D) -- (C) -- (A);
	\draw[color=red, thick] (H) -- (K) -- (R) -- (Q) -- (M) -- (H);
	\draw[color=red, thick] (L) -- (O) -- (U) -- (V) -- (P) -- (L);
	\draw[color=red, thick] (X) -- (AC) -- (AE) -- (AD) -- (Z) -- (X);
	\draw[color=red, thick] (AG) -- (AH) -- (AK) -- (AJ) -- (AI) -- (AG);
\end{tikzpicture}
\caption{The support $\tau$-tilting poset for $\bar\Lambda$ where $\Lambda = k(1\to 2)$.}
\label{fig:big}
\end{figure}
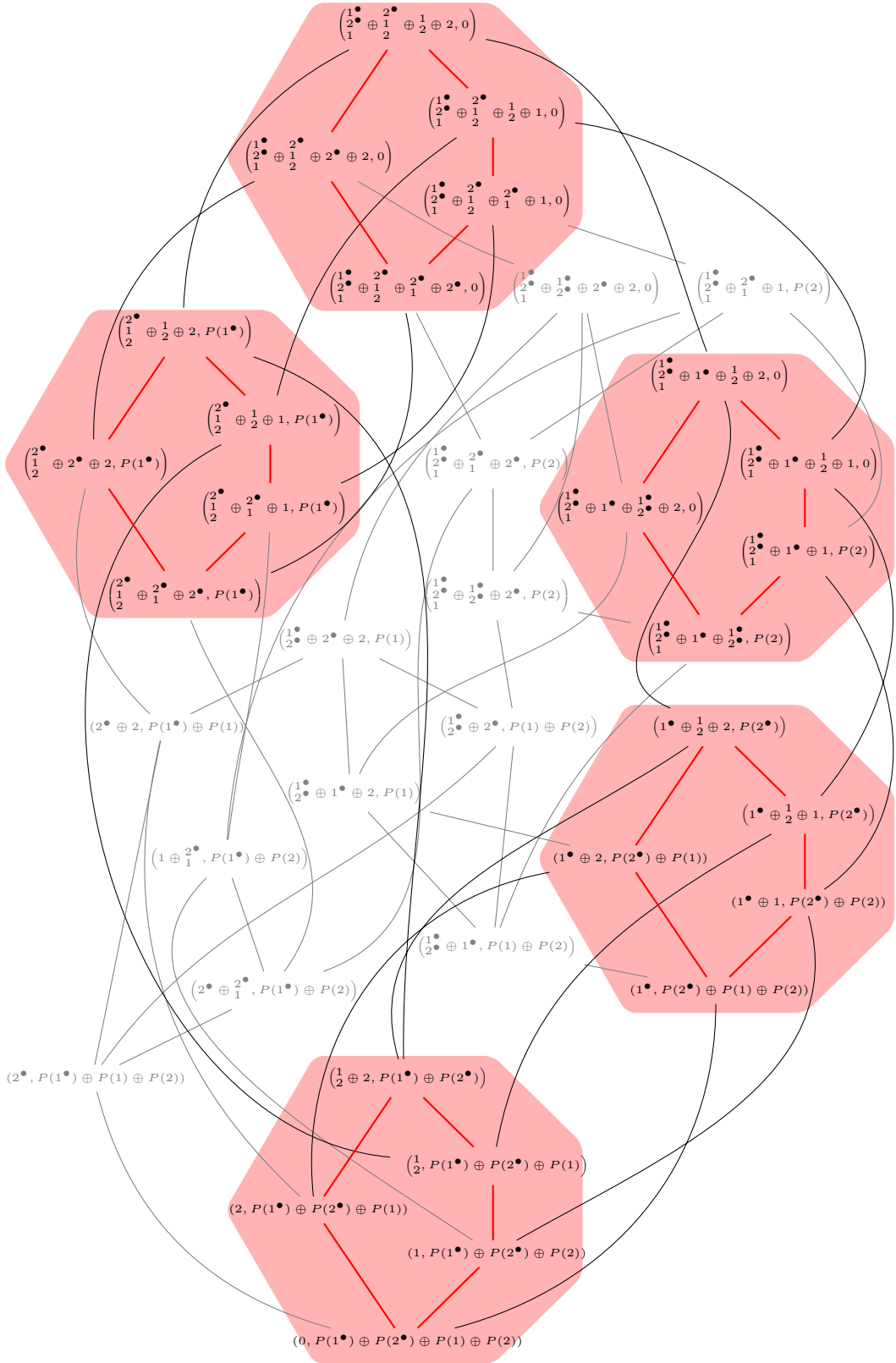

Next, we provide a relation between maximal green sequences in $\mod\,\Lambda$ and $\mod\,\bar\Lambda$.  Recall that a \emph{maximal green sequence} $\mathcal{M}=(\mathcal{M}_1, \dots, \mathcal{M}_t)$ is a finite maximal chain in $\stilt\,\Lambda$.  In particular, $\mathcal{M}$ starts at $\mathcal{M}_1=(\Lambda,0)$, ends at $\mathcal{M}_t=(0, \Lambda)$, and for any two consecutive elements of $\mathcal{M}$ we have $\mathcal{M}_i > \mathcal{M}_{i+1}$ is a covering relation in $\stilt\,\Lambda$.  Maximal green sequences were first introduced by Keller in \cite{K} in the context of cluster algebras and later generalized to module categories of finite dimensional algebras, for example see the survey \cite{KD} and references therein.  One of the main questions in this topic is whether an algebra admits a maximal green sequence. 

Let $\MGS(\Lambda)$ denote the set of maximal green sequences for $\Lambda$.  Given a sequence $\mathcal{M}$ of elements in $\stilt\,\Lambda$, which is not necessarily a maximal green sequence, and a map $g:\stilt\,\Lambda\to \stilt\,\bar\Lambda$, we define $g(\mathcal{M}) = (g(\mathcal{M}_1), \dots, g(\mathcal{M}_t))$ to be the resulting sequence.  Similarly, given two sequences  $\mathcal{M}=(\mathcal{M}_i)_{i=1}^t, \mathcal{M'}=(\mathcal{M}'_i)_{i=1}^{t'}$ with $\mathcal{M}_1=\mathcal{M}'_{t'}$ we define $\mathcal{M}\circ  \mathcal{M'}$, to be the composition $(\mathcal{M}'_1, \dots, \mathcal{M}'_{t'}= \mathcal{M}_1, \mathcal{M}_2, \dots, \mathcal{M}_{t})$.

With this notation we have the following result, which in particular says that the image of the map $\varphi$ from Corollary~\ref{cor:cross} spans the entire length of the support $\tau$-tilting poset of $\bar\Lambda$.

\begin{thm}
There is an inclusion of maximal green sequences 
\begin{align*}
\psi: \,\,\,  & \MGS(\Lambda) \times \MGS(\Lambda)  \to \MGS(\bar\Lambda)\\
 & (\mathcal{M}, \mathcal{M}')  \mapsto  \bar F_{0}(\mathcal{M'})\circ\bar G (\mathcal{M}).
\end{align*}
\end{thm}

\begin{proof}
Let $\mathcal{M}, \mathcal{M}'\in \MGS(\Lambda)$. Then by definition of $\bar G$ we have that 
\[\bar G (\mathcal{M}) = (\bar G(\Lambda, 0), \dots, \bar G (0, \Lambda))= ((\bar\Lambda, 0), \dots (\Lambda, Q_{\bar\Lambda}))\]
where $Q_{\bar\Lambda}$ denotes the projective-injective $\bar\Lambda$-modules.  Furthermore, by Theorem~\ref{cor:g} the sequence $\bar G (\mathcal{M})$ is a maximal chain in the interval 
\[ [(\bar\Lambda, 0), (\Lambda, Q_{\bar\Lambda})]\subset \stilt\,\bar\Lambda.\]

Similarly, 
\[\bar F_0(\mathcal{M}')= (\bar F_0(\Lambda, 0), \dots, \bar F_0 (0, \Lambda))\subset \stilt_{G(0,\Lambda)}\,\bar\Lambda\]
where by definition $\bar F_0(\Lambda, 0)$ is the Bongartz completion of $G(0, \Lambda)=(0,Q_{\bar\Lambda})$, which equals $(\Lambda, Q_{\bar\Lambda})$.  Moreover, $\bar F_0 (0, \Lambda)$ is the Bongartz co-completion of $G(0, \Lambda)=(0, Q_{\bar\Lambda})$, which equals $(0, \bar\Lambda)$.  Also note that by definition $\bar F_0$ is a poset isomorphism between $\stilt\,\Lambda$ and the Bongartz interval of $G(0, \Lambda)$, thus $\bar F_0(\mathcal{M}')$ is a maximal chain in the interval 
\[ [(\Lambda, Q_{\bar\Lambda}), (0, \bar\Lambda)]\subset \stilt\,\bar\Lambda.\]
This means that the composition $\bar F_{0}(\mathcal{M'})\circ\bar G (\mathcal{M})$ is a maximal green sequence for $\bar\Lambda$.  It is easy to see that $\psi$ is injective because both $\bar F_0$ and $\bar G$ are injective. 
\end{proof}

\begin{example}
Let $\Lambda=k(1\to 2)$ with support $\tau$-tilting poset shown in Example~\ref{ex:4}.  Consider two maximal green sequence for $\Lambda$ shown below. 
\[\mathcal{M} = \left(\big({\begin{smallmatrix}1\\2\end{smallmatrix}}\oplus {\begin{smallmatrix}2\end{smallmatrix}},{\begin{smallmatrix}0\end{smallmatrix}}\big),
\big({\begin{smallmatrix}1\\2\end{smallmatrix}}\oplus {\begin{smallmatrix}1\end{smallmatrix}},{\begin{smallmatrix}0\end{smallmatrix}}\big),
\big({\begin{smallmatrix}1\end{smallmatrix}},{\begin{smallmatrix}P(2)\end{smallmatrix}}\big),
\big( {\begin{smallmatrix}0\end{smallmatrix}},{\begin{smallmatrix}P(1)\oplus P(2)\end{smallmatrix}}\big)
\right)\]

\[\mathcal{M'} = \left(\big({\begin{smallmatrix}1\\2\end{smallmatrix}}\oplus {\begin{smallmatrix}2\end{smallmatrix}},{\begin{smallmatrix}0\end{smallmatrix}}\big),
\big({\begin{smallmatrix}2\end{smallmatrix}},{\begin{smallmatrix}P(1)\end{smallmatrix}}\big),
\big( {\begin{smallmatrix}0\end{smallmatrix}},{\begin{smallmatrix}P(1)\oplus P(2)\end{smallmatrix}}\big)
\right)\]

Then $\psi(\mathcal{M}, \mathcal{M'})$ is the following maximal green sequence for $\bar\Lambda$. 

\begin{align*}
\Big(\big({\begin{smallmatrix}1^\bullet\\2^\bullet\\1\;\;\end{smallmatrix}}\oplus {\begin{smallmatrix}2^\bullet\\1\;\;\\2\;\;\end{smallmatrix}}\oplus{\begin{smallmatrix}1\\2\end{smallmatrix}}\oplus {\begin{smallmatrix}2\end{smallmatrix}},{\begin{smallmatrix}0\end{smallmatrix}}\big),
\big({\begin{smallmatrix}1^\bullet\\2^\bullet\\1\;\;\end{smallmatrix}}\oplus {\begin{smallmatrix}1^\bullet\end{smallmatrix}}\oplus{\begin{smallmatrix}1\\2\end{smallmatrix}}\oplus {\begin{smallmatrix}2\end{smallmatrix}},{\begin{smallmatrix}0\end{smallmatrix}}\big),
\big({\begin{smallmatrix}1^\bullet\end{smallmatrix}}\oplus{\begin{smallmatrix}1\\2\end{smallmatrix}}\oplus {\begin{smallmatrix}2\end{smallmatrix}},{\begin{smallmatrix}P(2^\bullet)\end{smallmatrix}}\big),
\big({\begin{smallmatrix}1\\2\end{smallmatrix}}\oplus {\begin{smallmatrix}2\end{smallmatrix}},{\begin{smallmatrix}P(1^\bullet)\end{smallmatrix}}\oplus {\begin{smallmatrix}P(2^\bullet)\end{smallmatrix}}\big),\\
\big({\begin{smallmatrix}2\end{smallmatrix}}, {\begin{smallmatrix}P(1^\bullet)\oplus P(2^\bullet)\oplus P(1)\end{smallmatrix}}\big), 
\big(0, {\begin{smallmatrix}P(1^\bullet)\oplus P(2^\bullet)\oplus P(1)\oplus P(2)\end{smallmatrix}}\big)
 \Big)
 \end{align*}
 
Note that in Figure~\ref{fig:big} this sequence $\psi(\mathcal{M}, \mathcal{M'})$ corresponds to a sequence starting with $\bar G(\mathcal{M})$, that goes between the top-most elements in each highlighted pentagon, followed by a maximal sequence $\bar F_0(\mathcal{M}')$ in the highlighted pentagon at the bottom of the diagram.   
\end{example}

\end{document}